\documentclass[11pt]{amsart}

\usepackage[french,english]{babel}
\usepackage[utf8]{inputenc}
\usepackage[dvips]{graphicx}
\usepackage{calc}
\usepackage{color}
\usepackage{amsmath}
\usepackage{amssymb}
\usepackage{amscd}
\usepackage{amsthm}
\usepackage{amsbsy}
\usepackage{delarray}
\usepackage{enumerate}
\usepackage[T1]{fontenc}
\usepackage{inputenc}
\usepackage{enumerate}
\usepackage[all]{xy}
\usepackage{hyperref}

\begin{document}

\setlength{\parindent}{5mm}
\renewcommand{\leq}{\leqslant}
\renewcommand{\geq}{\geqslant}
\newcommand{\Z}{\mathbb{Z}}
\newcommand{\R}{\mathbb{R}}
\newcommand{\C}{\mathbb{C}}
\newcommand{\eps}{\varepsilon}
\newcommand{\chzrel}{c_{\mathrm{LR}}}
\newcommand{\Href}{H_{\mathrm{ref}}}
\newcommand{\Jref}{J_{\mathrm{ref}}}
\newcommand{\p}{\textbf{p}}

\theoremstyle{plain}
\newtheorem{theo}{Theorem}
\newtheorem{prop}[theo]{Proposition}
\newtheorem{lemma}[theo]{Lemma}
\newtheorem{definition}[theo]{Definition}
\newtheorem*{notation*}{Notation}
\newtheorem*{notations*}{Notations}
\newtheorem{corol}[theo]{Corollary}
\newtheorem{conj}[theo]{Conjecture}
\newtheorem*{claim*}{Claim}

\newenvironment{demo}[1][]{\addvspace{8mm} \emph{Proof #1.
    ~~}}{~~~$\Box$\bigskip}

\newlength{\espaceavantspecialthm}
\newlength{\espaceapresspecialthm}
\setlength{\espaceavantspecialthm}{\topsep} \setlength{\espaceapresspecialthm}{\topsep}

\newtheorem{exple}[theo]{Example}
\renewcommand{\theexple}{}
\newenvironment{example}{\begin{exple}\rm }{\hfill $\blacktriangleleft$\end{exple}}

\newtheorem{quest}[theo]{Question}
\renewcommand{\thequest}{}
\newenvironment{question}{\begin{quest}\it }{\end{quest}}

\newenvironment{remark}[1][]{\refstepcounter{theo} 
\vskip \espaceavantspecialthm \noindent \textsc{Remark~\thetheo
#1.} }%
{\vskip \espaceapresspecialthm}

\def\bb#1{\mathbb{#1}} \def\m#1{\mathcal{#1}}
\def\del{\partial}
\def\co{\colon\thinspace}
\def\id{\mathrm{Id}}
\def\Crit{\mathrm{Crit}}
\def\Spec{\mathrm{Spec}}
\def\osc{\mathrm{osc}}

\title[Reduction of symplectic homeomorphisms]{Reduction of symplectic homeomorphisms}
\author{Vincent Humili\`ere, R\'emi Leclercq, Sobhan Seyfaddini}
\date{\today}

\address{VH: Institut de Math\'ematiques de Jussieu, Universit\'e Pierre et Marie Curie, 4 place Jussieu, 75005 Paris, France}
\email{vincent.humiliere@imj-prg.fr}

\address{RL: Universit\'e Paris-Sud, D\'epartement de Math\'ematiques, Bat. 425, 91400 Orsay, France}
\email{remi.leclercq@math.u-psud.fr}

\address{SS: D\'epartement de Math\'ematiques et Applications de l'\'Ecole Normale Sup\'erieure, 45 rue d'Ulm, F 75230 Paris cedex 05}
\email{sobhan.seyfaddini@ens.fr}


\selectlanguage{english}
\begin{abstract}
 In \cite{HLS13}, we proved that symplectic homeomorphisms preserving a coisotropic submanifold $C$, preserve its characteristic foliation as well.  As a consequence, such symplectic homeomorphisms descend to the reduction of the coisotropic $C$.
 
  In this article we show that these reduced homeomorphisms continue to exhibit certain symplectic properties.  In particular, in the specific setting where the symplectic manifold is a torus and the coisotropic is a standard subtorus, we prove that the reduced homeomorphism preserves spectral invariants and hence the spectral capacity.
  
  To prove our main result, we use Lagrangian Floer theory to construct a new class of  spectral invariants  which satisfy a non-standard triangle inequality.

\bigskip
\selectlanguage{french}
\noindent\textsc{Résumé.}  Nous avons démontré dans \cite{HLS13}, qu'un homéomorphisme symplectique qui laisse invariante une sous-variété coisotrope $C$, préserve également son feuilletage caractéristique. Il induit donc un homéomorphisme sur la réduction symplectique de $C$. 
 
Dans cet article, nous démontrons que l'homéomorphisme ainsi obtenu exhibe certaines propriétés symplectiques. En particulier, dans le cas où la variété symplectique ambiante est un tore, et la sous-variété coisotrope est un sous-tore standard, nous démontrons que l'homéomorphisme réduit préserve les invariants spectraux et donc aussi la capacité spectrale. 
  
  Pour démontrer notre résultat principal, nous construisons, à l'aide de l'homologie de Floer Lagrangienne, une nouvelle famille d'invariants spectraux qui satisfont un nouveau type d'inégalité triangulaire.
 \end{abstract}

\selectlanguage{english}

\maketitle

\tableofcontents

\section{Introduction}
\label{sec:introduction}

\subsection{Context and main result}
\label{sec:context-main-result}

The main objects under study in this paper are symplectic homeomorphisms. Given a symplectic manifold $(M,\omega)$, a homeomorphism $\phi:M\to M$ is called a symplectic homeomorphism if it is the $C^0$--limit of a sequence of symplectic diffeomorphisms. This definition is motivated by a celebrated theorem due to Gromov and Eliashberg which asserts that if a symplectic homeomorphism $\phi$ is smooth, then it is a symplectic diffeomorphism in the usual sense: $\phi^\ast\omega=\omega$.  
 
  Understanding the extent to which symplectic homeomorphisms behave like their smooth counterparts constitutes the central theme of $C^0$--symplectic geometry.  
A recent discovery of Buhovsky and Opshtein suggests that these homeomorphisms are capable of exhibiting far more flexibility than symplectic diffeomorphisms: In \cite{buhovsky-opshtein}, they construct an example of a symplectic homeomorphism of the standard $\bb C^3$ whose restriction to the symplectic subspace $\bb C\times\{0\}\times\{0\}$ is the contraction $(z,0,0)\mapsto(\frac12z,0,0)$. Such behavior is impossible for a symplectic diffeomorphism but of course very typical for a volume-preserving homeomorphism.  On the other hand, it is well-known that symplectic homeomorphisms are surprisingly rigid in comparison to volume-preserving maps.  The following example of rigidity is the starting point of this article: Recall that a coisotropic submanifold is a submanifold $C\subset M$ whose tangent space, at every point of $C$, contains its symplectic orthogonal: $TC^\omega\subset TC$.  Moreover, the distribution $TC^\omega$ is integrable and the foliation  it spans is called the characteristic foliation of $C$. 

\begin{theo}[\cite{HLS13}]Let $C$ be a smooth  coisotropic submanifold of a symplectic manifold $(M,\omega)$. Let $\phi$ denote a symplectic homeomorphism. If $C'= \phi(C)$ is smooth, then it is coisotropic.  Furthermore, $\phi$ maps the characteristic foliation of $C$ to that of $C'$.
\end{theo}

 Prior to the discovery of the above theorem, the special cases of Lagrangian submanifolds and hypersurfaces have been treated, respectively, by Laudenbach--Sikorav \cite{LS94} and Opshtein \cite{opshtein}.\\

We are now in position to describe the problem we are interested in.  Denote by 
$\m F$ and $\m F',$ respectively, the characteristic foliations of the coisotropics $C$ and $C'$ from the above theorem.  The reduced spaces $\m R= C/\m F$ and $\m R'= C'/\m F'$ are defined as the quotients of the coisotropic submanifolds by their characteristic foliations.   These spaces are, at least locally, smooth manifolds and they can be equipped with natural symplectic structures induced by $\omega$.  Since $\phi(\m F) = \m F'$, the homeomorphism $\phi$ induces a homeomorphism $\phi_R: \m R \rightarrow \m R'$ of the reduced spaces. It is a classical fact that when $\phi$ is smooth, and hence symplectic, the reduced map $\phi_R$ is a symplectic diffeomorphism as well.
It is therefore natural to ask whether the homeomorphism $\phi_R$ remains symplectic, in any sense, when $\phi$ is not assumed to be smooth.  This is the question we seek to answer in this article.

We begin by first supposing that the reduction $\phi_R$ is smooth. It turns out that this scenario can be resolved rather easily using a result of \cite{HLS13}. 

\begin{prop}\label{prop:smooth-reduced-diffeo-is-sympl}
  Let $C$ be a coisotropic submanifold whose reduction $\m R$ is a symplectic manifold\footnote{This is always locally true.}, and $\phi$ be a symplectic homeomorphism. Assume that $C'=\phi(C)$ is smooth and therefore is coisotropic and admits a reduction $\m R'$. Denote by $\phi_R:\m R\to\m R'$ the map induced by $\phi$. Then, if $\phi_R$ is smooth, it is symplectic.
\end{prop}

We would like to point out that a similar result, with a similar proof, has already appeared in \cite{buhovsky-opshtein} (See Proposition 6). 

\begin{proof} We will prove that for any smooth function $f_R$ on $\m R'$, the Hamiltonian flow generated by the function $f_R\circ\phi_R$ is $\phi_R^{-1}\phi_{f_R}^t\phi_R$, where $\phi_{f_R}^t$ is the Hamiltonian flow generated by $f_R$.  It is not hard to conclude from this that $\phi_R$ is symplectic: For example, it can easily be checked that $\phi_R$ preserves the Poisson bracket, i.e.  $\{h_R\circ\phi_R, g_R \circ\phi_R \} = \{h_R, g_R\} \circ \phi_R$ for any two smooth functions $h_R$, $g_R$ on  $\m R'$.

Let $f_R \co \m R' \rightarrow \bb R$ be smooth. We denote by $g_R \co \m R \rightarrow \bb R$ the function defined by $g_R=f_R \circ \phi_R$. Let $f$ and $g$ be any smooth lifts to $M$ of $f_R$ and $g_R,$ respectively.

First, notice that by definition the restrictions to $C$ of $f \circ \phi$ and $g$ coincide. Since $g$ is constant on the characteristic leaves of $C$, its Hamiltonian flow $\phi_g^t$ preserves $C$.
Thus $H=(f\circ \phi - g)\circ \phi_g^t$ vanishes on $C$ for all $t$. By \cite[Theorem 3]{HLS13}, the flow of the continuous Hamiltonian\footnote{The continuous function $H$ generates a continuous flow in the sense defined by M\"uller and Oh \cite{muller-oh}.}$H$ follows the characteristic leaves of $C$.
On the other hand we know that this flow is given by the formula 
$\phi_H^t=(\phi_g^{t})^{-1} \phi^{-1} \phi_{f}^t\phi$.
This isotopy descends to the reduction $\m R$  where it induces the isotopy $(\phi_{g_R}^t)^{-1}\phi_R^{-1} \phi_{f_R}^t\phi_R$. But since $\phi_H^t$ follows characteristics it must descend to the identity. Hence $(\phi_{g_R}^t)^{-1}\phi_R^{-1} \phi_{f_R}^t\phi_R=\id$ as claimed.
\end{proof}

When $\phi_R$ is not assumed to be smooth, the situation becomes far more complicated.  The question of whether or not $\phi_R$ is a symplectic homeomorphism seems to be very difficult and, at least currently, completely out of reach.  Given the difficulty of this question, one could instead ask if there exist symplectic invariants which are preserved by reduced homeomorphisms.  In this spirit, and since symplectic homeomorphisms are capacity preserving, Opshtein formulated the following a priori easier problem:

\begin{question}\label{question:capacity-preserving} Is the reduction $\phi_R$ of a symplectic homeomorphism $\phi$ preserving a coisotropic submanifold always a capacity preserving homeomorphism?
\end{question}

  Partial positive results have been obtained by Buhovsky and Opshtein \cite{buhovsky-opshtein}. They proved in particular that in the case where $C$ is a hypersurface, the map $\phi_R$ is a ``non-squeezing map'' in the sense that for every open set $U$ containing a symplectic ball of radius $r$, the image $\phi_R(U)$ cannot  be embedded in a symplectic cylinder over a 2--disk of radius $R<r$. 
This does not resolve Opshtein's question, but since capacity preserving maps are non-squeezing it does provide positive evidence for it.  In the case of general coisotropic submanifolds, they conjecture that the same holds and  indicate as to how one might approach this conjecture.

In this article, we work in the specific setting where $M$ is the torus $\bb  T^{2(k_1+k_2)}$ equipped with its standard symplectic structure and  $C=\bb T^{2k_1+ k_2}\times\{0\}^{k_2}$.  The reduction of $C$ is $\bb T^{2k_1}$ with its usual symplectic structure.  Our main theorem shows that, in this setting, the reduced homeomorphism $\phi_R$ preserves certain symplectic invariants referred to as spectral invariants.  This answers Opshtein's question positively, as it follows immediately that the spectral capacity is preserved by $\phi_R.$

 More precisely, for a time-dependent Hamiltonian $H$, denote by $c_+(H)$ the spectral invariant, defined by Schwarz \cite{schwarz}, associated to the fundamental class of $M$.  Roughly speaking, $c_+(H)$ is the action value at which the fundamental class $[M]$ appears in the Floer homology of $H$; see Equation \eqref{eq:definition-cplus} in Section \ref{sec:hamilt-floer-theory} for the precise definition.  (We should caution the reader that our notations and conventions are different than those of \cite{schwarz}.  For example,  $c_+(H)$ in this article corresponds to  $c(1;H)$ in \cite{schwarz}  where 1 is the generator of $H^0(M)$.)\footnote{In \cite{schwarz}, after constructing $c(1;H)$ the author proceeds to normalize the Hamiltonian $H$  by requiring that $\int_{M} H(t,x) \omega^n = 0$ for each $t\in [0,1]$.  This leads to an invariant of Hamiltonian diffeomorphisms, $c(1;\phi_H^1)$.}  For degenerate or continuous functions one defines $c_+(H) = \lim_{i\to \infty}c_+(H_i)$ where $H_i$ is a sequence of smooth non-degenerate Hamiltonians converging uniformly to $H$.  This limit is well-defined because $c_+$ satisfies a well-known Lipschitz estimate. We refer the reader to Section \ref{sec:hamilt-floer-theory} for further details. Here is our main result:

\begin{theo}\label{theo:main-theo} Let $\phi$ be a symplectic homeomorphism of the torus $\bb  T^{2(k_1+k_2)}$ equipped with its standard symplectic form. Assume that $\phi$ preserves the coisotropic submanifold $C=\bb T^{2k_1+k_2}\times\{0\}^{k_2}$. Denote by $\phi_R$ the induced homeomorphism on the reduced space $\m R=\bb T^{2k_1}$. Then, for every time-dependent continuous function $H$ on $[0,1]\times \m R$, we have:
$$c_+(H\circ \phi_R)=c_+(H),$$
 where $H \circ \phi_R(t,x) := H(t, \phi_R(x)).$
\end{theo}

Note that Theorem \ref{theo:main-theo} implies that other related symplectic invariants which are constructed using spectral invariants are also preserved by $\phi_R$.  Here is one example of such invariants:  Following Viterbo [\cite{viterbo1}, Definition 4.11],  we define the spectral capacity of an open set $U$, denoted by $c(U)$,  by
$$c(U)=\sup\{c_+(H)\,|\, H \in C^0([0,1] \times M) , \; \textrm{support}(H_t)\subset U\ \; \forall t\in [0,1] \},$$ where   $C^0([0,1] \times M)$ denotes the space of time-dependent continuous functions on $M$.  The following is an immediate corollary of Theorem \ref{theo:main-theo}.
\begin{corol}
 The map $\phi_R$, from Theorem \ref{theo:main-theo}, preserves the spectral capacity, i.e. $c(\phi_R(U)) = c(U)$ for any open set $U$.
\end{corol}

In  Definition 5.15 of \cite{schwarz}, Schwarz defines a very similar capacity which he denotes by $c_\gamma$.  It can easily be checked that $\phi_R$ preserves $c_\gamma$ as well.

\subsection{Main Tools: Lagrangian Floer theory and  spectral invariants} \label{sec:intro_tech}
For proving Theorem \ref{theo:main-theo}, we will use the theory of Lagrangian spectral invariants.  These invariants were first introduced by Viterbo \cite{viterbo1} in the setting of cotangent bundles and using generating functions.  In \cite{Oh97}, Oh reconstructed the same invariants using Lagrangian Floer homology. There have been many developments in the theory since then; see Section \ref{sec:preliminaries-Floer} for specific references.

In this article, we will use Lagrangian Floer homology, in the specific setting where the symplectic manifold and the Lagrangians are all tori, to construct a new class of spectral invariants.  Below, we  describe our settings and give a brief overview of the construction and properties of the particular spectral invariants which will be used in the proof of Theorem \ref{theo:main-theo}.

The symplectic manifold we will be working on is the product
$$M = \bb T^{2k_1} \times \bb T^{2k_1} \times \bb T^{2k_2} \times \bb T^{2k_2}.$$
We denote by  $(q_1, p_1)$ and $(Q_1, P_1)$  the coordinates on the first and second $\bb T^{2k_1}$ factors in the above product, respectively.  The coordinates $(q_2, p_2)$ and $(Q_2, P_2)$ are defined similarly.  We equip $M$ with the standard symplectic structure given by  $$ \omega_{\mathrm{std}}= dq_1 \wedge dp_1 + dQ_1 \wedge dP_1 + dq_2 \wedge dp_2 + dQ_2 \wedge dP_2 \,.$$ 
The Lagrangian submanifolds of $M$ whose Floer homology we will be studying are  
\begin{align*}
  L_0 &= \bb T^{k_1} \times \{ 0 \} \times \bb T^{k_1} \times \{ 0 \} \times \bb T^{k_2} \times \{ 0 \} \times \bb T^{k_2} \times \{ 0 \} \,, \\
  L_1 &= \bb T^{k_1} \times \{ 0 \} \times \bb T^{k_1} \times \{ 0 \} \times \bb T^{k_2} \times \{ 0 \} \times \{ 0 \} \times \bb T^{k_2} \,.
\end{align*}
Notice that both $L_0$ and $L_1$ decompose as  products of smaller Lagrangians, i.e. $L_i = L \times L_i'$, where
\begin{align*}
  L = &\;\bb T^{k_1} \times \{ 0 \} \times \bb T^{k_1} \times \{ 0 \} \times \bb T^{k_2} \times \{ 0 \}  \;\subset\; \bb T^{2k_1} \times \bb T^{2k_1} \times \bb T^{2k_2} \,,\\
  &L_0' =  \bb T^{k_2} \times \{ 0 \} \;\subset\; \bb T^{2k_2},  \mbox{ and }\;   L_1' = \{ 0 \} \times \bb T^{k_2} \;\subset\; \bb T^{2k_2} \,.
 \end{align*}
Observe that $L_0 \cap L_1 = L \times \{ 0 \}.$  In Section \ref{sec:spec-floer-theory}, we will construct an isomorphism between the Morse homology of $L$, denoted by $HM(L)$, and the Floer homology group $HF(L_0, L_1)$;  see Theorem \ref{theo:Floer-hom-L0L1} for a precise statement.  We will then use this isomorphism to associate  a critical value of the Lagrangian action functional $\m A^{L_0,L_1}_H$ to a non-zero class $a \in HM(L)$ and a Hamiltonian $H:[0,1] \times M \rightarrow \R$.  We will denote this critical value by  $$\ell(a; L_0, L_1; H).$$
This is the spectral invariant associated to $a$ and $H$.  Roughly speaking, $\ell(a; L_0, L_1; H)$ is the action value at which the Morse homology class $a$ appears in the Floer homology group $HF(L_0, L_1)$.

\subsubsection*{Main properties of spectral invariants}
\label{sec:main-prop-spectr}

We now list some of the main properties of the spectral invariant $\ell(a; L_0, L_1; H).$

\medskip

\noindent \textbf{1. Spectrality:} Let  $\mathrm{Spec}(H)$ denote the set of critical values of the action functional $\m A^{L_0,L_1}_H.$  Then, for any Hamiltonian $H$ and $a \in HM(L) \setminus \{0\},$
 $$\ell(a; L_0, L_1; H) \in  \mathrm{Spec}(H).$$
For further details, see Section \ref{sec:spectrality}.

\medskip

\noindent \textbf{2. Continuity:} The following inequality holds for any Hamiltonians $H, H'$ 
\begin{align*}
   \int_0^1 \min_M (H_t-H'_t) \,dt \leq |\ell(a;L_0,L_1;H)-&\ell(a;L_0,L_1;H')| \\
&\qquad \leq \int_0^1 \max_M (H_t-H'_t) \,dt \,. 
\end{align*}
For further details, see Sections \ref{sec:lagr-spectr-invar} and \ref{sec:spec-floer-theory}.

\medskip

\noindent \textbf{3. Splitting Formula:} 
Let $F$ and $F'$ denote two Hamiltonians on $\bb T^{2k_1} \times \bb T^{2k_1} \times \bb T^{2k_2}$ and $\bb T^{2k_2}$, respectively.  Define the Hamiltonian $F \oplus F'$ on $M$ by $F \oplus F' (z_1, z_2)= F(z_1) + F'(z_2),$ for $z_1 \in \bb T^{2k_1} \times \bb T^{2k_1} \times \bb T^{2k_2}$ and $z_2 \in  \bb T^{2k_2}.$ 
In Section \ref{sec:prdct_formula_l0l1}, we obtain the following ``splitting'' formula:
$$\ell(a; L_0, L_1; F \oplus F') = \ell(a; L, L ; F) + \ell([\mathrm{pt}]; L_0', L_1'; F'),$$
where $\ell(a; L, L ; F)$ denotes the standard Lagrangian spectral invariant associated to $a \in HM(L)$ and $\ell([\mathrm{pt}]; L_0', L_1'; F')$ denotes the spectral invariant associated to the only non-zero class in $HF(L_0', L_1')$.  See Sections \ref{sec:lagr-spectr-invar} and \ref{sec:SI-in-case-single-lagr} for the definitions of $\ell([\mathrm{pt}]; L_0', L_1'; F')$ and $\ell(a; L, L ; F)$, respectively.  Section \ref{sec:when-l_0cap-l_1} provides further details on $HF(L_0', L_1')$.

\medskip

\noindent \textbf{4. Triangle Inequalities:} Given two Hamiltonians $H, H'$ we denote by $H\#H'$ the Hamiltonian whose flow is the concatenation of the flows of $H$ and $H';$ see Equation (\ref{eq:defn_sharp}) for a precise definition of $H\#H'$.  Consider two Morse homology classes $a, b \in HM(L)$ such that the intersection product $a \cdot b \neq 0$. Lastly, for $i=0,1$, denote by $[L_i']$ the fundamental class in $HM(L_i')$.  Then, the following triangle inequalities hold:
\begin{align*}
\ell(a\cdot b; L_0, L_1; H \# H') \leq  \ell(a ; L_0, L_1; H)+ \ell(b \otimes [L_1']; L_1, L_1;  H'),\\
\ell(a\cdot b; L_0, L_1; H \# H') \leq  \ell(a \otimes [L_0']; L_0, L_0 ;H)  + \ell(b; L_0, L_1;  H').
\end{align*}

\medskip

The first three of the above properties are more or less standard, and in fact, we prove these in a more general setting in Section \ref{sec:Lagr-Floer-theory}.  The fourth property, which is perhaps the most interesting one, is specific to our settings and is different than the triangle inequality which appears  in the standard setting where only one Lagrangian is considered. 

Proofs of triangle inequalities of this nature consist of two main steps. First, one must prove a purely Floer theoretic version of the triangle inequality where Morse homology classes and the Morse intersection product are replaced with their Floer theoretic analogues. We do this, in a more general setting than what is described here in the introduction, in Section \ref{sec:lagr-pdct-triangle-ineq};  see Theorem \ref{theo:triangle_inequality_general}.  The second step  involves establishing a correspondence between the Morse and Floer theoretic versions of the intersection product, the latter being usually referred to as the pair-of-pants product.  It is well-known that when $L_0$ and $L_1$ coincide (and some technical assumptions are satisfied) the two versions of the intersection product coincide up to a PSS-type isomorphism; see Equation \eqref{sec:pop_intersec_prdct}.  In our case, however, such a direct correspondence does not exist; the pair-of-pants product is not even defined on the tensor product of a single ring!
In Theorem \ref{theo: prdct_struct_L0L1} and Remark \ref{rem:full_desc_prdct}, we fully describe the relation between the intersection product on $HM(L)$ and the pair-of-pants products $*: HF(L_0, L_1) \otimes HF(L_1, L_1) \rightarrow HF(L_0, L_1)$ and $*: HF(L_0, L_0) \otimes HF(L_0, L_1) \rightarrow HF(L_0, L_1)$.
  
\subsubsection*{Comparing the two forms of spectral invariants}  
 Using the aforementioned properties of the spectral invariants $ \ell(a; L_0, L_1; H),$ one can deduce several other interesting properties of these invariants.  Here, we will mention a comparison result which plays a significant role in our proof of Theorem \ref{theo:main-theo}.
  
 Denote by $\ell([L_i]; L_i, L_i; H)$ the standard Lagrangian spectral invariant associated to the fundamental class $[L_i] \in HM(L_i)$; see Section \ref{sec:SI-in-case-single-lagr} for the definition.  The triangle inequality allows us to compare the two forms of spectral invariants.  More precisely, we prove the following in Section \ref{sec:traingle_ineq_l0l1}.
  
\begin{prop}\label{cor:triangle-inequality-1}
For $i=0,1,$ denote by $[L_i]$ the fundamental classes in $HM(L_i).$  Then, for any non-zero $a\in HM(L)$ and any Hamiltonian $H$:
$$\ell(a; L_0, L_1; H) \leq \ell([L_i]; L_i, L_i; H).$$
In particular, $\ell([L]; L_0, L_1; H) \leq \ell([L_i]; L_i, L_i; H),$ where $[L] \in HM(L)$ is the fundamental class.
\end{prop}

\begin{remark} In defining the above spectral invariants $\ell(a;L_0,L_1;H)$, we were inspired by the construction of ``conormal spectral invariants'' defined in a cotangent bundle $T^*N$ via consideration of the Lagrangian Floer homology of the zero section $0_N$ and the conormal $\nu^*V$ of a submanifold $V\subset N$ (see e.g. \cite{Oh97}). Indeed, if we heuristically think of the torus $\bb T^{2k_1} \times \bb T^{2k_1} \times \bb T^{2k_2} \times \bb T^{2k_2}$ as a compact version of the cotangent bundle to $\bb T^{k_1} \times \bb T^{k_1} \times \bb T^{k_2} \times \bb T^{k_2}$, then the Lagrangian $L_0$ corresponds to the zero section and $L_1$ corresponds to the conormal bundle of the submanifold $V=\bb T^{k_1} \times \bb T^{k_1} \times \bb T^{k_2} \times \{0\}$. 

Of the above four properties of the spectral invariants $\ell(a;L_0,L_1;H)$, the first three also hold for conormal spectral invariants.  We believe that, by readjusting the techniques used in this paper, one could obtain an appropriately reformulated version of the triangle inequality for conormal spectral invariants.  
This would then lead to the following comparison inequalities, corresponding to Proposition \ref{cor:triangle-inequality-1}: For every homology class $a\in HM(V)$, and every Hamiltonian $H$ on $T^*N$,
$$\ell(a;0_N,\nu^*V;H)\leq \ell([N];0_N,0_N;H).$$
As far as we know, the triangle inequality has not yet been proven for conormal spectral invariants.  However, the above comparison inequalities were proven, via generating-function techniques, in \cite{viterbo1}.

  The idea that conormal spectral invariants could be useful in studying the behavior of spectral invariants under symplectic reduction has been present in works based on generating function theory (e.g.  \cite{theret}, \cite{humiliere1}, \cite{ST}) and goes back to Viterbo \cite{viterbo1}.  
  To the best of our knowledge, this article is the first place where this idea is implemented in Floer theory.  We found this implementation to be necessary for our purposes as Floer theory is better suited for working on compact manifolds.
\end{remark}

\subsection*{Organization of the paper}
In Sections \ref{sec:Lagr-Floer-theory}--\ref{sec:comp-lagr-ham-si} we recall Floer theoretic preliminaries, define Lagrangian and Hamiltonian spectral invariants, and prove some of their essential properties.  In Section \ref{sec:lagr-pdct-triangle-ineq}, we define the pair-of-pants product and prove a purely Floer theoretic version of the triangle inequality in a fairly general setting.  In Section \ref{sec:kunn-form-prod}, we prove a K\"unneth formula for Lagrangian Floer homology and use it to derive a splitting formula for spectral invariants.
In Section \ref{sec:lag_Floer_tori}, we specialize the Floer theory of Section  \ref{sec:preliminaries-Floer} to the specific settings introduced above.  In Section \ref{sec:prdct_triangle_l0l1}, we prove the aforementioned triangle inequality.
Lastly, in Section \ref{sec:proof}, we use the results from Sections \ref{sec:preliminaries-Floer} and \ref{sec:lag_Floer_tori} to prove Theorem \ref{theo:main-theo}.

\subsection*{Aknowledgements}
We are grateful to Claude Viterbo for several helpful conversations.

This work is partially supported by ANR Grants ANR-11-JS01-010-01 and ANR-13-JS01-0008-01.  The research leading to these results has received funding from the European Research Council under the European Union's Seventh Framework Programme (FP/2007-2013) / ERC Grant Agreement  307062. 

%
%

\section{Floer homology and spectral invariants}
\label{sec:preliminaries-Floer}

\subsection{Lagrangian Floer homology}\label{sec:Lagr-Floer-theory}

In this section, we review the construction of Floer homology. Throughout the section, we fix a closed symplectic manifold $(M,\omega)$, two closed non-disjoint connected Lagrangian submanifolds $L_0$, $L_1$ and $p\in L_0\cap L_1$ an intersection point.  Recall that 
\begin{itemize}
\item $(M,\omega)$ is symplectically aspherical if $\omega|_{\pi_2(M)}=0$,
\item a Lagrangian $L$ of $(M,\omega)$ is weakly exact, or the pair $(M,L)$ is relatively symplectically aspherical, if $\omega|_{\pi_2(M,L)}=0$.
\end{itemize}
We say that the pair $(L_0,L_1)$ is \textit{weakly exact with respect to $p$}, if  any disk in $M$ whose boundary is on $L_0\cup L_1$ and is ``pinched'' at $p$ has vanishing symplectic area. More precisely, denote by $D$ the unit disk in $\bb C$ centered at $0$. Denote by $\del D^+$ the upper half of $\del D$, $\del D^+=\{ z\in \bb C : |z|=1, \mathrm{Im}(z) \geq 0  \}$, and by $\del D^-$ its lower half. 
\begin{definition}\label{def:weak_exact}
   The pair of Lagrangians $(L_0, L_1)$ is weakly exact with respect to $p \in L_0\cap L_1$ if for any map $u \co (D,\del D^+,\del D^-,\{-1,1\}) \rightarrow (M,L_0,L_1,\{ p \})$, $\int_{D} u^* \omega =0$.
\end{definition}
Notice that in this case both $L_0$ and $L_1$ are weakly exact and thus $M$ is symplectically aspherical. 

\begin{example}\label{ex:weakly_exact}
The Lagrangians we will consider in Sections \ref{sec:lag_Floer_tori} and \ref{sec:proof} form weakly exact pairs with respect to any point in their intersections. Recall from Section \ref{sec:intro_tech} in the introduction that for $i=1$ and $2$, $L_i = L\times L'_i$ are Lagrangians of $(\bb T^{k}\times \bb T^l,\omega_{\bb T^k}\oplus \omega_{\bb T^l})$ so that $\bb T^l=L'_0 \times L'_1$ and $L'_0 \cap L'_1 = \{ 0 \}$. (In this example only, $k$ and $l$ denote the respective integers $4k_1+2k_2$ and $2k_2$ to ease the reading.) We fix a point $p=(p_k,0)$ in $L_0 \cap L_1$. 

First notice that since $L$ is a subtorus of $\bb T^k$, $\pi_2(\bb T^k,L)=0$ so that $(L,L)$ is a weakly exact pair with respect to $p_k$. 

Next, consider a pinched disk
$$u \co (D,\del D^+,\del D^-,\{-1,1\}) \rightarrow (\bb T^l, L'_0 \times \{ 0 \},  \{ 0 \} \times L'_1,\{ 0 \}) \,.$$
Denote by $\gamma_0$ and $\gamma_1$ the loops respectively in $L'_0$ and $L'_1$, defined by $u(\del D^-)=\gamma_0 \times \{ 0 \}$ and $u(\del D^+)= \{ 0 \} \times \gamma_1$. Since $[u(\del D^-)]=[u(\del D^+)] \in \pi_1(\bb T^l)=\pi_1(L'_0) \times \pi_1(L'_1)$, $\gamma_0$ and $\gamma_1$ are null-homotopic in $L'_0$ and $L'_1$ respectively. By gluing to $u$ two disks $v_i \subset L'_i$ bounding $\gamma_i$, we obtain a sphere in $\bb T^l$ whose symplectic area necessarily vanishes. Since the $v_i$'s are included in Lagrangians we deduce that $\omega_{\bb T^l}(u)=0$. Therefore $(L'_0,L'_1)$ is weakly exact with respect to the single intersection point $0$.

Finally, since the product of weakly exact pairs is weakly exact, we deduce that $(L_0,L_1)$ is weakly exact with respect to $p$.
\end{example}

Let $H \co [0,1] \times M \rightarrow \bb R$ be a smooth Hamiltonian function. We will denote by $X_H$ the 1--parameter family of vector fields induced by $H$ by $\omega(X_H^t, \cdot\,)=-dH_t$ for all $t$ and by $\phi_H^t$ its flow satisfying: $\phi_H^0=\id$ and for all $t$, $\del_t \phi_H^t=X_H^t(\phi_H^t)$. We first consider a \textit{non-degenerate} Hamiltonian, which means in this case that the intersection $\phi_H^1(L_0)\cap L_1$ is transverse. A generic Hamiltonian is non-degenerate.

We denote by $\Omega(L_0,L_1;p)$ the set of paths $x$ from $L_0$ to $L_1$ which are in the connected component of the constant path $p$. Such a path admits a capping $\bar x \co [0,1]\times[0,1] \rightarrow M$ so that: For all $t \in [0,1]$, $\bar x(0,t)=p$ and $\bar x(1,t)=x(t)$, $[0,1]\times \{0\}$ is mapped to $L_0$ and $[0,1]\times \{1\}$ to $L_1$.

Two cappings $\bar x_1$ and $\bar x_2$ of $x \in \Omega(L_0,L_1;p)$ have the same symplectic area since $\bar x_1 \# (-\bar x_2)$ is a pinched disk as defined above so that it has area 0 by assumption. (Recall that $-\bar x_2$ stands for $\bar x_2$ with reverse orientation.) Thus we can define the action functional by the formula:
\begin{align}
\label{eq:action-funct}
  \m A^{L_0,L_1}_H \co \Omega(L_0,L_1;p) \rightarrow \bb R \,, \qquad x \mapsto -\int \bar{x}^*\omega + \int_0^1 H_t(x(t)) \,dt \,.
\end{align}

The critical points of $\m A^{L_0,L_1}_H$ are paths $x \in \Omega(L_0,L_1;p)$ which are orbits of $H$ that is for all $t$, $x(t)=\phi_H^t(x(0))$. These orbits are in one-to-one correspondence with $\phi_H^1(L_0)\cap L_1$ so that their number is finite since $H$ is non-degenerate and $M$ compact. One defines the Floer complex as the $\bb Z_2$--vector space $CF(L_0,L_1;p;H)=\langle \Crit(\m A^{L_0,L_1}_H) \rangle_{\bb Z_2}$. The set of critical values of $\m A^{L_0,L_1}_H$ is called its spectrum and is denoted by $\mathrm{Spec}(H)$. 

Now Floer's differential is defined thanks to perturbed pseudo-holomorphic strips: we pick a 1--parameter family of tame, $\omega$--compatible, almost complex structures $J$. We define the set of Floer trajectories between two orbits of $H$, $x_-$ and $x_+$, as 
\begin{align*}
  \m{\widehat{M}}^{L_0,L_1}(x_-,x_+;H,J) = \left\{\! u \co \bb R \times [0,1] \rightarrow M \left|\!
      \begin{array}{l}   \del_s u + J_t(u)(\del_t u - X_H^t(u))=0 \\ \forall t,\, u(\pm\infty,t)= x_\pm(t) \\ u(\bb R \times \{ 0 \}) \subset L_0 \\ u(\bb R \times \{ 1 \}) \subset L_1 \end{array}
\!\right. \!\!\right\}
\end{align*}
where the limits $u(\pm\infty,t)$ are uniform in $t$. There is an obvious $\bb R$--action by reparametrization $s \mapsto s+\tau$ and we define $\m M^{L_0,L_1}(x_-,x_+;H,J)$ as the quotient $\m{\widehat{M}}^{L_0,L_1}(x_-,x_+;H,J)/\bb R$.

Requiring that the pair $(H,J)$ is \emph{regular}, that is the linearization of the operator $\overline{\del}_{J,H} \co u \mapsto \del_s u+J_t(u)(\del_t u - X_H(u))$ is surjective for all $u \in \widehat{\m M}^{L_0,L_1} (x_-,x_+;H,J)$, ensures that $\m M^{L_0,L_1}(x_-,x_+;H,J)$ is a smooth manifold.  We denote its 0-- and 1--dimensional components respectively by $\m M^{L_0,L_1}_{[0]}(x_-,x_+;H,J)$ and $\m M^{L_0,L_1}_{[1]}(x_-,x_+;H,J)$.

\begin{remark}
 It turns out that we do not need to consider \textit{graded} complexes. As a consequence, we do not mention the different standard indices usually entering into play in such theories. In particular, we do not require additional assumptions concerning these indices. 

 Nonetheless, let us recall that the $(i+1)$--dimensional component of $\widehat{\m M}^{L_0,L_1}(x_-,x_+;H,J)$ consists of those Floer trajectories whose Maslov--Viterbo index equals $i+1$, see e.g \cite{Au13}. When there are no such trajectories, we put $\m M^{L_0,L_1}_{[i]}(x_-,x_+;H,J)$ to be the emptyset.
\end{remark}

The 0--dimensional component of the moduli space of all Floer trajectories running between any two orbits, $\m M^{L_0,L_1}_{[0]}(H,J)= \cup_{x_-,x_+} \m M^{L_0,L_1}_{[0]}(x_-,x_+;H,J)$ is compact. Floer's differential is defined by linearity on $CF(L_0,L_1;p;H)$ after setting the image of a generator as
\begin{align*}
  \del^{L_0,L_1}_{H,J} (x_-) = \sum_{x_+} \# \m M^{L_0,L_1}_{[0]}(x_-,x_+;H,J) \cdot x_+
\end{align*}
where $\# \m M$ is the mod 2 cardinal of $\m M$ and the sum runs over all orbits $x_+$. Since the asphericity assumption prevents bubbling of disks and spheres, by Gromov's compactness Theorem and standard gluing $\m M^{L_0,L_1}_{[1]}(H,J)$, the 1--dimensional component of $\m M^{L_0,L_1}(H,J)$, can be compactified, and this fact ensures that $(\del^{L_0,L_1}_{H,J})^2=0$ that is, $\del^{L_0,L_1}_{H,J}$ is a differential.

The Floer homology of the pair $(L_0,L_1)$ is the homology of this complex $HF(L_0,L_1;p;H,J)=H(CF(L_0,L_1;p;H),\del^{L_0,L_1}_{H,J})$. Because it is often useful to keep in mind the specific Floer data which we used to define the complex, we will keep $(H,J)$ in the notation, however the homology does not depend on the choice of the regular pair $(H,J)$.\footnote{This being said, when there is no risk of confusion we will denote $HF(L_0,L_1;p;H,J)$ by $HF(L_0,L_1;p)$ to simplify the notation.} Indeed, there are morphisms
\begin{align*}
  \Psi_{H,J}^{H',J'} \co CF(L_0,L_1;p;H) \rightarrow CF(L_0,L_1;p;H')
\end{align*}
inducing isomorphisms in homology which are called continuation isomorphisms. Roughly, $\Psi_{H,J}^{H',J'}$ is defined thanks to a regular homotopy between $(H,J)$ and $(H',J')$, $(\tilde H,\tilde J)$, by considering the 0--dimensional component of the moduli space of Floer trajectories for the pair $(\tilde H,\tilde J)$ running from an orbit of $H$ to an orbit of $H'$ with boundary condition on $L_0$ and $L_1$ respectively. It is also standard---and the proof is based on the same principle by considering a homotopy between homotopies---that $\Psi_{H,J}^{H',J'}$ does not depend on the choice of the homotopy $(\tilde H,\tilde J)$. From these facts, one gets that they are ``canonical'', that is they satisfy
\begin{align}\label{eq:compos_cont_maps}
  \Psi_{H,J}^{H,J}=\id \quad \mbox{and} \quad \Psi_{H,J}^{H',J'}\circ \Psi_{H',J'}^{H'',J''}=\Psi_{H,J}^{H'',J''}
\end{align}
for any three regular pairs $(H,J)$, $(H',J')$, and $(H'',J'')$. \\

We now present two situations which will be of particular interest to us and in which one can actually compute Floer homology.

\subsubsection{The case of a single Lagrangian}
\label{sec:when-l_0=l_1}

Assume that $L_0$ and $L_1$ coincide and denote $L=L_0=L_1$. Assume moreover that $L$ is connected. In that case, the assumption that the pair $(L,L)$ is weakly exact with respect to any given point $p \in L$ is equivalent to requiring the Lagrangian $L$ to be weakly exact.

 It is well-known that there exists an isomorphism between the Floer homology of $(L,L)$ and the Morse homology of $L$, called PSS isomorphism. It was defined in the Hamiltonian setting by Piunikhin--Salamon--Schwarz \cite{PSS}, then adapted to Lagrangian Floer homology by Kati\'c--Milinkovi\'c \cite{KaticMilinkovic} for cotangent bundles, and by Barraud--Cornea \cite{BarraudCornea06} and Albers \cite{Albers} for compact manifolds. For details, we refer to Leclercq \cite{Leclercq08} which deals with weakly exact Lagrangians in compact manifolds which is the situation we are interested in here.

The PSS morphism requires an additional choice of a Morse--Smale pair $(f,g)$, consisting of a Morse function $f \co L \rightarrow \bb R$ and a metric $g$ on $L$. It is defined at the chain level $\Phi^{L}_{H,J} \co CM(L;f,g) \rightarrow CF(L,L;p;H, J)$ by counting the number of elements of suitable moduli spaces. It commutes with the differential and thus induces a morphism in homology:
\begin{align*}
   \Phi^{L}_{H,J} \co HM(L) \rightarrow HF(L,L;p;H, J)
\end{align*}
(as the notation suggests, we will omit the Morse data). It is an isomorphism and its inverse $\Phi_{L}^{H,J} \co HF(L,L;p; H, J) \rightarrow HM(L)$ is defined in the same fashion. The main properties of the PSS morphism which will be needed are the following two:
\begin{enumerate}
  \item PSS morphism commutes with continuation morphisms, that is
    \begin{align}
    \begin{split}
    \label{eq:PSS-comm-continuation}
        \xymatrix@C=1.2cm{\relax
           HM(L) \ar[r]^{\hspace{-.8cm}\Phi^{L}_{H,J}} \ar[rd]_{\Phi^{L}_{H',J'}} &  HF(L,L;p;H,J) \ar[d]^{\Psi_{H,J}^{H',J'}} \\
          & HF(L,L;p;H'J') 
        }
    \end{split}
    \end{align}
    commutes for any two regular pairs $(H,J)$ and $(H',J')$.
  \item PSS morphism intertwines the Morse and Floer theoretic versions of the intersection product in homology, the latter being known as pair-of-pants product, see subsection \ref{sec:pop_intersec_prdct} for the precise statement.
\end{enumerate}

\subsubsection{The case of two Lagrangians intersecting transversely at a single point}
\label{sec:when-l_0cap-l_1}
When $L_0$ and $L_1$ intersect transversely at a single point $p$, the Hamiltonian $H=0$ is non-degenerate. The associated Floer complex obviously has a single generator, $p$ itself. Moreover, for any choice of almost complex structure $J$ such that $(0,J)$ is regular, the boundary map is trivial since the 0--dimensional component of the moduli space of Floer trajectories from $p$ to itself is empty.
 It follows that $HF(L_0, L_1;p;0,J)$, and hence $HF(L_0, L_1;p;H,J)$ for any regular $(H,J)$, is isomorphic to the group with two elements. We will refer to this isomorphism, which is uniquely defined, as a PSS-type morphism and will denote it by
\begin{align*}  
  \Phi^{L_0,L_1}_{H,J}:\Z_2\to HF(L_0,L_1;p;H,J).
\end{align*}
 The only non-zero class in $HF(L_0, L_1;p;H,J)$ will be denoted by $[\mathrm{pt}]$. 

Since there exists only one isomorphism between two given groups with two elements, the following diagram commutes for any two regular pairs $(H,J)$ and $(H',J')$:
\begin{align}
    \begin{split}
    \label{blaaah}
        \xymatrix@C=1.2cm{\relax
           \Z_2 \ar[r]^{\hspace{-1.4cm}\Phi^{L_0, L_1}_{H,J}} \ar[rd]_{\hspace{-1.6cm}\Phi^{L_0, L_1}_{H',J'}} &  HF(L_0,L_1;p;H,J) \ar[d]^{\Psi_{H,J}^{H',J'}} \\
          & HF(L_0,L_1;p;H',J') 
        }
    \end{split}
    \end{align}

\subsection{Lagrangian spectral invariants}
\label{sec:lagr-spectr-invar}

Spectral invariants for Lagrangians in cotangent bundles were introduced by Viterbo \cite{viterbo1} using generating functions. This was adapted to Floer homology by  Oh \cite{Oh97}. Since then there have been several extensions to other settings. See in particular Leclercq \cite{Leclercq08} for a single weakly-exact Lagrangian and Zapolsky \cite{zapolsky} for a weakly exact pair of Lagrangians intersecting at a single point. 

We provide below a new extension of the definition for a general weakly-exact pair $(L_0,L_1)$ with a given intersection point $p$.

To give this definition, the starting observation is the standard fact that for every Floer trajectory $u\in \m M^{L_0,L_1}(x_-,x_+;H,J)$, 
$$\m A_H^{L_0,L_1}(x_-)-\m A_H^{L_0,L_1}(x_+)=\int_{\R\times[0,1]}\|\partial_su\|^2dsdt \geq 0,$$
where $\|\cdot\|$ is the norm associated to the metric $\omega(\cdot,J\cdot)$. Thus the action decreases along Floer trajectories. 
Now let $H$ be a non-degenerate Hamiltonian and let $a\in\R$ be a regular value of the action functional, i.e. $a\notin\mathrm{Spec}(H)$. It follows from this observation that if $CF^a(L_0,L_1;p;H)$ denotes the $\Z_2$--vector space generated by Hamiltonian chords of action $<a$, then $CF^a(L_0,L_1;p;H)$
is a subcomplex of $CF(L_0,L_1;p;H)$. We denote $i^a:HF^a(L_0,L_1;p;H,J) \to HF(L_0,L_1;p;H,J)$ the map induced in homology by the inclusion.
For every non-zero Floer homology class $\alpha\in HF(L_0,L_1;p;H,J)$, we define the spectral invariant associated to $\alpha$ to be the number
\begin{align}\label{def:lag_spec_inv_1}
\ell(\alpha;L_0,L_1;p;H)= \inf \{ a \in \bb R : \alpha \in \mathrm{im}(i^a) \} \,.
\end{align}

We wish to have the ability to compare the spectral invariants of different Hamiltonians. For this purpose it will be convenient to fix a reference Floer homology group: pick a regular pair $(\Href,\Jref)$ and set $$HF_{\mathrm{ref}}(L_0,L_1;p)=HF(L_0,L_1;p;\Href,\Jref) \,.$$

\begin{definition}\label{def:lag_spec_inv}  For every non-degenerate Hamiltonian function $H$, the spectral invariant associated to $\alpha\in HF_{\mathrm{ref}}(L_0,L_1;p)$, $\alpha \neq 0$, is the number
$$\ell(\alpha;L_0,L_1;p;H):=\ell(\Psi_{\Href,\Jref}^{H,J}(\alpha);L_0,L_1;p;H) \,.$$
\end{definition}

As the notation suggests, the number $\ell(\alpha;L_0,L_1;p;H)$, both in the above definition and in Equation \eqref{def:lag_spec_inv_1}, does not depend on the necessary choice of an almost complex structure $J$ so that the pair $(H,J)$ is regular. This follows from the following inequality which holds for every two regular pairs $(H,J)$, $(H',J')$:
\begin{align}\label{eq:ineq-PSS}
\begin{split}
\ell(\Psi_{\Href,\Jref}^{H',J'}(\alpha);L_0,L_1;p;H') &\leq \ell(\Psi_{\Href,\Jref}^{H,J}(\alpha);L_0,L_1;p;H) \\
&\qquad\qquad\qquad + \int_0^1 \max_M (H'_t-H_t) \, dt  \,.
\end{split}
\end{align} 
We now sketch a proof of the above inequality. Since the continuation morphism is injective, the image of any non-zero class $\alpha \in HF(L_0,L_1;p;H,J)$ is non-zero. By definition of $\Psi_{H,J}^{H',J'}$ there exist Floer trajectories for a homotopy $(\tilde{H},\tilde{J})$ between the generators of $CF(L_0,L_1;p;H)$ whose linear combination represents $\alpha$ and the generators of $CF(L_0,L_1;p;H')$ representing $\Psi_{H,J}^{H',J'}(\alpha)$. Computing the energy of such a trajectory and using the fact that the result is positive yields:
$$ \ell(\Psi_{H,J}^{H',J'}(\alpha);L_0,L_1;p;H') \leq \ell(\alpha;L_0,L_1;p;H) + \int_0^1 \max_M (H'_t-H_t) \, dt  \,.$$
Thus, in particular Inequality \eqref{eq:ineq-PSS} follows.

Furthermore, Inequality (\ref{eq:ineq-PSS}) implies that for every non-degenerate $H$, $H'$, 
\begin{align}\label{eq:unif_continuity} \begin{split}
   \int_0^1 \min_M (H_t-H'_t) \,dt &\leq |\ell(\alpha;L_0,L_1;p;H)-\ell(\alpha;L_0,L_1;p;H')| \\
&\qquad\qquad\qquad\qquad\qquad \leq \int_0^1 \max_M (H_t-H'_t) \,dt \,. \end{split}
\end{align}
As a consequence, the number $\ell(\alpha;L_0,L_1;p;H)$ is Lipschitz continuous with respect to the Hamiltonian $H$.  Moreover, it follows that $\ell(\alpha;L_0,L_1;p;H)$ can be defined by continuity for every continuous function $H:[0,1]\times M\to\R$.

\subsubsection{Spectrality}
\label{sec:spectrality}

It is rather easy in our situation (where one does not need to keep track of cappings) to prove the \textit{spectrality property} of the invariants $\ell(\alpha;L_0,L_1;p;H)$ regardless the non-degeneracy of $H$. Namely, for all non-zero $\alpha \in HF_{\mathrm{ref}}(L_0,L_1;p)$, 
$$\ell(\alpha;L_0,L_1;p;H) \in \mathrm{Spec}(H) \,.$$
We will need the following consequence of this property (we refer to \cite[Lemma 2.2]{MVZ12} for a proof).
\begin{corol}\label{corol:H=c-then-l=c}
  Let $L$ be a weakly exact closed Lagrangian of $(M,\omega)$ and $H \co [0,1] \times M \rightarrow \bb R$ be continuous. If $H|_L = c$ for some $c \in \bb R$, then $\ell(\alpha;L,L;p;H) = c$ for all $\alpha \neq 0$ in $HF(L,L;p)$.
\end{corol}

We end this subsection by recalling that $\Spec(H)$, for any Hamiltonian $H$, is a measure zero subset of $\R$.  This fact will be used in Section \ref{sec:proof}.

\subsubsection{Naturality}
\label{sec:naturality}

Lagrangian Floer homology is \textit{natural} in the sense that for any symplectomorphism $\psi \co (M,\omega) \rightarrow (M',\omega')$ and any two Lagrangians $L_0$ and $L_1$ of $(M,\omega)$, the following Floer homologies are isomorphic 
\begin{align}
  \label{eq:naturality-HF}
  HF(L_0,L_1;p;H,J) \simeq HF(\psi(L_0),\psi(L_1);\psi(p);H\circ \psi^{-1},(\psi^{-1})^* J)
\end{align}
since the respective complexes as well as the respective moduli spaces involved in the definition of the differential are in one-to-one correspondence. This one-to-one correspondence is given on the generators of the complex by the obvious identification
 \begin{align*}
   x \in \mathrm{Crit}\big(\m A_H^{L_0,L_1}\big) & \Leftrightarrow   \psi(x) \in \mathrm{Crit}\big(\m A_{H\circ \psi^{-1}}^{\psi(L_0),\psi(L_1)}\big) 
 \end{align*}
where $\psi(x)$ denotes the orbit of $H\circ \psi^{-1}$ given as $t \mapsto \psi(x(t))$. Furthermore, the above bijection preserves the action, namely
\begin{align*}
  \forall x \in \mathrm{Crit}\big(\m A_H^{L_0,L_1}\big), \quad  \m A_H^{L_0,L_1}(x) = \m A_{H\circ \psi^{-1}}^{\psi(L_0),\psi(L_1)}(\psi(x)) \,.
\end{align*}

From this, it is easy to see that the respective Lagrangian spectral invariants coincide: For any non-zero Floer homology class $\alpha$ in $HF(L_0,L_1;p;H,J)$ and its image via \eqref{eq:naturality-HF}, $\alpha_\psi$ in $HF(\psi(L_0),\psi(L_1);\psi(p);H\circ \psi^{-1},(\psi^{-1})^* J)$, we have
\begin{align}
\label{eq:naturality-SI}
  \ell(\alpha;L_0,L_1;p;H) = \ell(\alpha_\psi;\psi(L_0),\psi(L_1);\psi(p);H\circ \psi^{-1}) \,.
\end{align}

\subsubsection{The case of a single Lagrangian}
\label{sec:SI-in-case-single-lagr}

In the particular situation of Section \ref{sec:when-l_0=l_1} where we consider a single Lagrangian $L$ ($=L_0=L_1$), one can easily associate spectral invariants not only to Floer homology classes of $(L,L)$ but also to (Morse) homology classes of $L$ via the PSS isomorphism. For convenience, we denote these invariants in the same way: To any $a \neq 0$ in $HM(L)$, we associate
\begin{align}  \label{eq:SI-L0=L1}
  \ell(a;L,L;H) = \ell(\Phi_{H,J}^L(a);L,L;p;H) 
\end{align}
with $p$ any point in $L$ and the right-hand side defined by \eqref{def:lag_spec_inv_1}. 

As in the general case, this quantity requires the additional choice of an almost complex structure $J$ such that $(H,J)$ is regular, it is Lipschitz continuous with respect to $H$, so that it is independent of the choice of $J$ and its definition naturally extends to any continuous $H \co [0,1]\times M  \rightarrow \bb R$.

The naturality \eqref{eq:naturality-SI} of spectral invariants also holds in this case. More precisely, it reads 
\begin{align}
\label{eq:naturality-SI-2}
  \ell(a;L,L;p;H) = \ell(\psi_*(a);\psi(L),\psi(L);\psi(p);H\circ \psi^{-1})
\end{align}
for all non-zero homology classes $a$ of $L$ and all symplectomorphisms $\psi$. Here $\psi_*(a)$ denotes the image of $a$ by the morphism induced by $\psi|_L$ on $HM(L)$. Indeed, to see that this holds one should pick a Morse--Smale pair $(f,g)$ on $L$, and use $(f\circ \psi^{-1},(\psi^{-1})^* g)$ as Morse--Smale pair on $\psi(L)$. For these choices, the following diagram commutes:
\begin{align*} 
  \xymatrix{\relax
    HM(L;f,g) \ar[r]^{\hspace{-1.3cm}\psi_*=(\psi|_L)_*} \ar[d]_{\Phi^L_{H,J}}  & HM(\psi(L);f\circ \psi^{-1},(\psi^{-1})^* g) \ar[d]^{\Phi^{\psi(L)}_{H\circ \psi^{-1},(\psi^{-1})^*J}} \\
    HF(L,L;p;H,J) \ar[r]_{\hspace{-2.0cm}\psi_*} & HF(\psi(L),\psi(L);\psi(p);H\circ \psi^{-1},(\psi^{-1})^* J)
  }
\end{align*}  
even at the chain level (this is a mild generalization of \cite[Lemma 5.1]{HLL11} where $\psi$ was assumed to be a Hamiltonian diffeomorphism preserving $L$). The fact that spectral invariants do not depend on the Morse data then leads to \eqref{eq:naturality-SI-2}.\\

Finally, in the case of a single Lagrangian one spectral invariant will be of particular interest to us, namely the one associated to $[L]$, the fundamental class of $L$: $\ell([L];L,L;H)$. 

\subsection{Hamiltonian Floer theory, spectral invariants, and capacity}
\label{sec:hamilt-floer-theory}

\subsubsection{Hamiltonian Floer homology}
\label{sec:hamilt-floer-homol}

We work in a symplectically aspherical manifold $(M,\omega)$. Formally, this case is very similar to the Lagrangian case of section \ref{sec:when-l_0=l_1}.

Namely, we pick a Hamiltonian $H$ which is non-degenerate in the sense that the graph of $\phi_H^1$, $\Gamma_{\phi_H^1}$, intersects transversely the diagonal $\Delta \subset M\times M$. 

Instead of $\Omega(L,L;p)$, we consider the set of contractible free loops in $M$. We denote this set by $\Omega(M)$ and a typical element by $\gamma$. The action functional $\m A_H \co \Omega(M) \rightarrow \bb R$ is defined by the same formula as \eqref{eq:action-funct} except that for $\gamma \in \Omega(M)$, $\bar\gamma$ denotes a capping of $\gamma$, that is a disk in $M$ whose boundary is mapped to the image of $\gamma$. Again, the asphericity condition ensures that $\m A_H$ is well-defined.

Its critical points are the contractible 1--periodic orbits of $H$ which form a finite set by genericity of $H$ and generate a $\bb Z_2$--vector space which we denote $CF(M;H)$. 

We again pick a 1--parameter family of tame, $\omega$--compatible, almost complex structures $J$ and consider the moduli spaces:
\begin{align*}
  \m{\widehat{M}}(\gamma_-,\gamma_+;H,J) = \left\{\! u \co \bb R \times S^1 \rightarrow M \left|\!
      \begin{array}{l}   \del_s u + J_t(u)(\del_t u - X_H^t(u))=0 \\ \forall t,\, u(\pm\infty,t)= \gamma_\pm(t)  \end{array}
\!\right. \!\!\right\}
\end{align*}
and their quotient by the obvious $\bb R$--reparametrization in $s$ which we denote $\m M(\gamma_-,\gamma_+;H,J)$. These moduli spaces share the same properties as their Lagrangian counterpart and the differential is defined accordingly:
\begin{align*}
  \del_{H,J}(\gamma_-) = \sum_{\gamma_+} \# \m M_{[0]}(\gamma_-,\gamma_+;H,J) \cdot \gamma_+
\end{align*}
on generators and extended by linearity. Again, the sum runs over all contractible 1--periodic orbits of $H$ and $\m M_{[0]}$ is the 0--dimensional component of the moduli space $\m M$. 

The Floer homology of $(M,\omega)$ is defined as the homology of this complex $HF(M)=H(CF(M;H), \del_{H,J})$ and does not depend on the choice of the regular pair $(H,J)$ in the sense that there are continuation isomorphisms defined in the exact same fashion as in the Lagrangian case and built via regular homotopies of the data. When $H$ is $C^2$--small enough, the Floer complex coincides with the Morse complex of $M$.

Finally, there is also a PSS morphism constructed from a regular pair $(H,J)$ and a Morse--Smale pair $(f,g)$ on $M$ similarly to its Lagrangian counterpart. As for the latter, we will omit the Morse data and denote it $\Phi_{H,J} \co HM(M) \rightarrow HF(M;H,J)$.

\subsubsection{Hamiltonian spectral invariants}
\label{sec:hamilt-spectr-invar}

This case corresponds to the one studied by Schwarz in \cite{schwarz}. As in the Lagrangian case described above, the Floer complex is naturally filtered by action values since the action functional decreases along Floer trajectories. So any regular value of the Hamiltonian action $\m A_H$ gives rise to a subcomplex $CF^a(M;H)  \stackrel{i^a}{\hookrightarrow} CF(M;H)$ and the Hamiltonian spectral invariants are defined for any non-zero Floer homology class of $M$ as in \eqref{def:lag_spec_inv_1}. Thanks to the PSS isomorphism, one can also associate spectral invariants to any non-zero (Morse) homology class of $ \alpha \in HM(M)$ as in Section~\ref{sec:SI-in-case-single-lagr}.  We will temporarily use the notation  $c(\alpha;H,J)$ to denote these invariants.

These invariants share similar properties with their Lagrangian counterparts.  In particular they satisfy a Lipschitz estimate similar to  \eqref{eq:unif_continuity}.   It follows that they are independent of the choice of almost complex structure and hence we will denote them $c(\alpha;H)$.  Furthermore, being Lipschitz continuous, $c(\alpha; H)$ extends to continuous functions on $[0,1]\times M$, i.e. for a continuous $H \in C^0([0,1]\times M)$ we can define $c(\alpha;H) = \lim_{i \to \infty} c(\alpha;H_i)$ where $H_i$ is any sequence of smooth non-degenerate Hamiltonians converging to $H$.

As in \ref{sec:SI-in-case-single-lagr}, one of these invariants will be of greatest interest to us, $c_+(H)=c([M];H)$, the Hamiltonian spectral invariant associated to the fundamental class of $M$.  It follows from the above discussion that for non-degenerate $H$, it is defined via the following expression:
\begin{align}\label{eq:definition-cplus}
c_+(H)= \inf \{ a \in \bb R : \mathrm{PSS}([M]) \in \mathrm{im}(HF^a(M;H) \stackrel{i^a_*}{\longrightarrow} HF(M;H)) \} \,.
\end{align}

\subsubsection{The spectral capacity $c$}
\label{sec:spectral-capacity-c}

Following Viterbo \cite{viterbo1}, we extract from $c_+$ the spectral capacity $c$ mentioned in the introduction. Namely, for any open set $U$ in $M$, we define
\begin{align*}
  c(U)=\sup  \{ c_+(H) : H \in C^0([0,1] \times M), \; \mathrm{supp}(H_t) \subset U \forall t \in [0,1]\} \,.
\end{align*}
This quantity satisfies the properties defining a capacity, see \cite{viterbo1}.

\subsection{Comparison between Lagrangian and Hamiltonian spectral invariants}
\label{sec:comp-lagr-ham-si}

There is an action-preserving isomorphism between the Hamiltonian Floer complex of $(M,\omega)$ associated to a regular pair $(H,J)$ and the Lagrangian Floer complex of the diagonal $\Delta \simeq M$ seen as a Lagrangian in $(M\times M,(-\omega) \oplus \omega)$ and associated to an appropriate regular pair $(\hat H,\hat J)$. The goal of this section is to prove that the respective spectral invariants coincide (see also \cite[Section 3.4]{Leclercq08}). Given Hamiltonians $H$ and $G$ on $M$, we will denote by $H \oplus G$ the Hamiltonian given for every $(x,y)\in M\times M$ by $H \oplus G(x,y)=H(x)+G(y)$.
\begin{prop}
\label{prop:comp-lagr-ham-si}
  Let $(M, \omega)$ be a symplectically aspherical manifold. Let $\alpha \neq 0$ in $HM(M)$ and denote $\hat \alpha$ the corresponding class in $HM(\Delta)$. For any continuous time-dependent Hamiltonian $H$ on $M$,  $c(\alpha;H)=\ell(\hat \alpha;\Delta,\Delta;0 \oplus H)$. In particular, $c_+(H)=\ell([\Delta];\Delta,\Delta; 0 \oplus H)$.
\end{prop}
Notice that we are in the case of a single Lagrangian submanifold $\Delta$, so that $\ell(\hat \alpha;\Delta,\Delta;0 \oplus H)$ refers to this particular setting, see Section \ref{sec:SI-in-case-single-lagr}.\\

At several points in this paper, and to begin with in the proof of the proposition above, we will need to work with Hamiltonians $H$ such that $H_t = 0$ for $t$ near $0$ and $1$.  This can always be achieved, without affecting the spectral invariants of $H$, by time reparametrization. This is the content of the following remark. 
\begin{remark} \label{rem:time_rep}
  Let $(H,J)$ be a regular pair. Pick a smooth increasing function $\sigma \co [0,1] \rightarrow \bb R$ so that $\sigma(t)=0$ for all $t \in [0,\varepsilon]$ and $\sigma(t)=1$ for all $t \in [1- \varepsilon' ,1]$ for some $\varepsilon$, $\varepsilon'$ so that $0< \varepsilon <1-\varepsilon'< 1$. Then define $H^\sigma_t(x)=\sigma'(t)H_{\sigma(t)}(x)$. There is an obvious bijection between the sets of orbits of $H$ and $H^\sigma$ which leads to a bijection on the Floer complexes as vector spaces:
\begin{align}\label{eq:isom-MVZ-reparam}
    CF(M;H) \rightarrow CF(M;H^\sigma) \,, \quad \gamma \mapsto [ \gamma^\sigma \co t \mapsto \gamma(\sigma(t)) ] 
\end{align}
which preserves the action, namely $\m A_{H^\sigma}(\gamma^\sigma)=\m A_H(\gamma)$ (since geometrically the orbits are the same, a capping of $\gamma$ also caps $\gamma^\sigma$). 

Then define $J^\sigma$ as $J^\sigma_t(x)=J_{\sigma(t)}(x)$. Notice that $(H^\sigma,J^\sigma)$ is regular, and that there is a bijection between the moduli spaces $\m M(\gamma_-,\gamma_+;H,J)$ and $\m M(\gamma^\sigma_-,\gamma^\sigma_+;H^\sigma,J^\sigma)$ so that \eqref{eq:isom-MVZ-reparam} induces an action-preser\-ving isomorphism of the differential complexes. Notice that geometrically the main objects (orbits and Floer's strips) remain the same. It is thus easy to see that geometrically the representatives of a given Floer homology class remain unchanged along the process so that, together with the fact that the action is preserved, the associated (Hamiltonian) spectral invariants coincide.

For the same reason, given a Lagrangian $L$, the Lagrangian spectral invariants associated to $H$ also remain unchanged along such reparametrization.
\end{remark}

We now prove Proposition \ref{prop:comp-lagr-ham-si}.
\begin{proof}
  First notice that if $(M,\omega)$ is symplectically aspherical, then the diagonal $\Delta$ is a weakly exact Lagrangian of $(M\times M,(-\omega)\oplus \omega)$.

We first prove the proposition for non-degenerate Hamiltonians. So we start with a regular pair $(H,J)$ and apply Remark \ref{rem:time_rep} with $\sigma \co [0,1] \rightarrow \bb R$ so that $\sigma(t)=0$ for all $t \in [0,1/2]$. Then $H^\sigma=0$ and $J^\sigma_t=J_0$ for all $t \in [0,1/2]$. 

Now we consider for $t \in [0,\tfrac{1}{2}]$ the Hamiltonian $\hat H_t = H^\sigma_{\frac{1}{2}-t} \oplus H^\sigma_{\frac{1}{2}+t}$ on $M\times M$. There is a bijection:
\begin{align}\label{eq:isom-lagr-ham}
  \begin{split}
    CF(M;H^\sigma) &\rightarrow CF(\Delta,\Delta; \hat H) \,,\\
[\gamma \co \bb S^1 \rightarrow M ] &\mapsto \left[ x \co \!\!\left[ 0,\tfrac{1}{2} \right] \rightarrow M\times M \right] \mbox{ with } x(t)=(\gamma(\tfrac12-t),\gamma(\tfrac{1}{2}+t))
  \end{split}
\end{align}
since $x$ is an orbit of $\hat H$ if and only if $\gamma$ is an orbit of $H^\sigma$. Notice that by definition of $\sigma$, $\hat H_t = 0 \oplus H^\sigma_{\frac{1}{2}+t}$ so that by Remark \ref{rem:time_rep} above 
\begin{align}
  \label{eq:time-rep-in-pdct}
  \ell([\Delta];\Delta,\Delta;\hat H)=\ell([\Delta];\Delta,\Delta;0\oplus H)
\end{align}
and $x(t)=(\gamma(0),\gamma(\tfrac{1}{2}+t))$.  

Again, we need an appropriate family of almost complex structures $\hat J$ on $M\times M$ which we obtained by putting $\hat J_t(x,y) = -J^\sigma_{\frac{1}{2}-t}(x) \times J^\sigma_{\frac{1}{2}+t}(y)$ for $t\in [0,\tfrac{1}{2}]$. It is easy to see that $(H^\sigma,J^\sigma)$ is regular and that the bijection \eqref{eq:isom-lagr-ham} above is compatible with the differentials of the complexes. Indeed, pick any two generators of $CF(M;H^\sigma)$, $\gamma_-$ and $\gamma_+$ and any cylinder $u \co \bb R \times \bb S^1$ which uniformly converges to $\gamma_\pm$ when $s$ converges to $\pm\infty$. Denote respectively by $x_\pm$ the generators of $CF(\Delta,\Delta;\hat H)$ given by \eqref{eq:isom-lagr-ham} from $\gamma_\pm$ and consider
\begin{align*}
  \hat u \co \bb R \times \left[ 0 , \tfrac{1}{2} \right] \rightarrow M \times M \,, \quad \hat u(s,t)=\left(u\left(s, \tfrac{1}{2} -t \right),u \left(s,\tfrac{1}{2}+t \right)\right) .
\end{align*}
When $s$ goes to $\pm\infty$, $\hat u$ uniformly converges to $(\gamma_{\pm}(\tfrac{1}{2} -t),\gamma_{\pm}(\tfrac{1}{2} +t))=x_{\pm}(t)$. The boundary conditions are: $\hat u(s,0)=(u(s,\tfrac{1}{2}),u(s,\tfrac{1}{2}))$ and $\hat u(s,\tfrac{1}{2})=(u(s,0),u(s,1))$ which both lie in $\Delta$ for any $s\in \bb R$. Finally, projecting Floer's equation 
\begin{align*}
\forall t \in [0,\tfrac{1}{2}],\;  \del_s \hat u + \hat J_t(\hat u) (\del_t \hat u - X_{\hat H}^t (\hat u)) = 0
\end{align*}
to both components of the product shows that it is satisfied if and only if
\begin{align*}
\forall t \in [0,1],\;  \del_s u +  J^\sigma_t(u) (\del_t u - X_{H^\sigma}^t (u)) = 0 \,.
\end{align*}
Thus $\hat u \in \m M^{\Delta,\Delta}(x_-,x_+;\hat H,\hat J)$ if and only if $u \in \m M(\gamma_-,\gamma_+;H^\sigma,J^\sigma)$ and \eqref{eq:isom-lagr-ham} induces an isomorphism of complexes.

Finally, remark that there is an obvious correspondence between the cappings of a 1--periodic orbit $\gamma$ and the half-cappings of its associated orbit $x$. In particular, a capping $\bar \gamma$ of $\gamma$ can be thought of as a half-capping for $x$, by putting $\bar x=(\bar x_1,\bar x_2) \co D^2 \rightarrow M\times M$, with $\bar x_1$ the constant half-capping mapping $D^2$ to $\gamma(0)$ and $x_2$ the half-capping mapping $\del D^+$ to the image of $\gamma$ and $\del D^-$ to $\gamma(0)$. By doing so, not only $\bar x$ maps $\del D^+$ to the image of $x$ in $M\times M$ and $\del D^-$ to $(\gamma(0),\gamma(0)) \in \Delta$, but the symplectic area of $\bar \gamma$ with respect to $\omega$ and the symplectic area of $\bar x$ with respect to $(-\omega)\oplus\omega$ coincide. It easily follows that the action is preserved along the above transformation, namely $\m A_{H^\sigma}(\gamma)=\m A^{\Delta,\Delta}_{\hat H}(x)$.\footnote{ To be perfectly precise, we should have used an additional time-reparametrization to define $\hat{H}$ on the whole interval $[0,1]$. Since such a reparametrization is harmless in terms of spectral invariants as explained in Remark \ref{rem:time_rep}, we omitted it.}

Now pick a Morse--Smale pair $(f,g)$ on $M$ and define $(\hat f,\hat g)$ on $\Delta$ by putting $\hat f(x,x)=f(x)$ and $\hat g_{(x,x)}((\xi,\xi),(\eta,\eta))=g_x(\xi,\eta)$ for all $x$ in $M$ and all $\xi$ and $\eta$ in $T_xM$. Then the pair $(\hat f,\hat g)$ is a Morse--Smale pair for $\Delta$ and it is easy to show that the moduli spaces involved in the definition of the Hamiltonian PSS morphism in $M$ correspond to those defining the Lagrangian PSS morphism in $M\times M$ with respect to $\Delta$ along the above process. Thus, for any non-zero homology class $\alpha \in HM(M)$, which we denote $\hat \alpha$ when seen as a homology class in $HM(\Delta)$, $c(\alpha;H^\sigma)=\ell(\hat \alpha;\Delta,\Delta;\hat H)$. In particular, when $\alpha=[M]$, $\hat\alpha = [\Delta]$ so that $c_+(H^\sigma)=\ell([\Delta];\Delta,\Delta;\hat H)$.

Combined with \eqref{eq:time-rep-in-pdct}, this concludes the proof for smooth non-degenerate Hamiltonians $H$. In view of the extension of both $c$ and $\ell$ to $C^0([0,1]\times M)$, Proposition \ref{prop:comp-lagr-ham-si} easily follows from the non-degenerate case.
\end{proof}

\subsection{Products in Lagrangian Floer theory and the triangle inequality}
\label{sec:lagr-pdct-triangle-ineq}

Let $L_0$, $L_1$, and $L_2$ denote three Lagrangian submanifolds of $(M,\omega)$.  We fix three intersection points $p_{01} \in L_0 \cap L_1$, $p_{12} \in L_1 \cap L_2$, $p_{02} \in L_0 \cap L_2$ and suppose that each pair $(L_i, L_j)$ is weakly exact, in the sense of Definition \ref{def:weak_exact}, with respect to the intersection point $p_{ij} \in L_i \cap L_j$.  In this section,  we describe the usual product structure on Lagrangian Floer homology.  

We will be closely following  the construction of this product as described in \cite[Section 3]{AS10}.  There exist several other ways of defining the same product; see for example \cite{Au13}.  Let $\Sigma$ denote the Riemann surface obtained by removing three points from the boundary of the closed unit disk in $\C$.  We view $\Sigma$ as a strip with a slit: 
$$\Sigma = (\R \times [-1, 0] \sqcup \R \times [0, 1] ) / \sim, $$ where $(s, 0^-) \sim (s, 0^+)$ for all $s \geq 0$.  This is indeed a Riemann surface whose interior is naturally identified with $\R \times (-1, 1)\, \setminus (-\infty, 0] \times \{0\}$ and whose boundary consists of the three components $ \R \times \{-1\}, \, \R \times \{1\},$ and $(-\infty, 0] \times \{0^-, 0^+\}.$  At any point, other than $(0,0)$, the inclusion of $\Sigma$ into $\C$ induces the standard complex structure $(s,t) \mapsto s + it.$  At the point $(0, 0)$ the complex structure is given by the map $\{z \in \C: \mathrm{Re}(z) \geq 0\} \to \Sigma, \, z \mapsto z^2$. 

\begin{figure}[h]
  \centering
  \includegraphics[width=6.92cm]{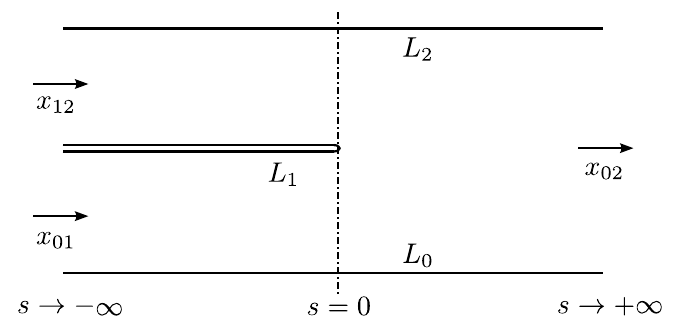}
  \caption{Abbondandolo--Schwarz's strip with a slit, $\Sigma$}
  \label{fig:split-strip}
\end{figure}

For $0\leq i < j \leq 2,$ denote by $(H_{ij},J_{ij})$ a regular pair (of a Hamiltonian and a compatible time-dependent almost complex structure) for the weakly exact pair of Lagrangians $(L_i, L_j)$. Without loss of generality, we may assume that $H_{ij}(t,x) = 0$ for $t$ near $0$ and $1$; see Remark \ref{rem:time_rep}.   To define the product structure, we need some auxiliary data:  For $s \in (-\infty, \infty)$ and $t \in [-1,1]$ let $J_{(s,t)}$ denote a family of almost complex structures on $M$ such that $$J_{(s,t)} = \left\{ \begin{array}{ll}  J^{t+1}_{01} & \mbox{ if } s\leq -1, \; t \in [-1,0], \\ J^t_{12} &  \mbox{ if } s \leq -1, \; t\in [0,1], \\ J_{02}^{\frac{t+1}{2}} &  \mbox{ if } s \geq 1, \; t\in [-1,1].  \end{array}\right.$$

Furthermore, Choose a function $K :  \R \times [-1, 1] \times M  \rightarrow \R$ such that 

$$K(s,t,x) = \left\{ \begin{array}{ll}  H_{01}(t+1,x)  &  \mbox{ if } s\leq -1, \; t \in [-1,0], \\ H_{12}(t,x) &  \mbox{ if } s \leq -1, \; t\in [0,1], \\ \frac{1}{2} H_{02}(\frac{t+1}{2},x) &  \mbox{ if } s \geq 1, \; t\in [-1,1].  \end{array}\right.$$

For any three Hamiltonian chords $x_{ij} \in CF(L_i,L_j;p_{ij};H_{ij})$, consider the moduli space $\m M(x_{01},x_{12}; x_{02})$ of maps  $u: \Sigma \to M$ solving the Floer-type equation $\del_s u + J_{(s,t)}(u)( \del_t u - X_{K}^{s,t}(u) )=0 $ and subject to the following asymptotic and boundary conditions
$$\begin{cases}  
\forall t\in[-1,0], \, u(-\infty,t)= x_{01}(t+1) \mbox{ and } \forall t\in[0,1], \, u(-\infty,t)= x_{12}(t) , \\   
 \forall t\in[-1,1], \, u(+\infty,t)= x_{02}\big(\frac{t+1}{2}\big), \\  
u(\R \times \{-1\}) \subset L_0, \, u(\R \times \{1\})\subset L_2, \, u((-\infty, 0] \times \{0^-, 0^+\}) \subset L_1. 
\end{cases}$$

For generic choices of $K$ and $J$, the moduli space  $\m M(x_{01},x_{12}; x_{02})$ is a smooth finite dimensional manifold.  Its  $0$--dimensional component, denoted by $\m M_{[0]}(x_{01},x_{12}; x_{02}),$ is compact and thus finite.  We denote by $\# \m M_{[0]}(x_{01},x_{12}; x_{02})$ its cardinality modulo $2$.   We can now define a bilinear map 
\begin{align*}
  CF(L_0,L_1;p_{01};H_{01}) \times &CF(L_1,L_2;p_{12};H_{12}) \rightarrow  CF(L_0,L_2;p_{02};H_{02}) \\
  (x_{01}, x_{12}) & \mapsto \sum_{x_{02}} {\# \m M_{[0]}(x_{01},x_{12}; x_{02})} \cdot x_{02} \,. 
\end{align*}
  This map depends on the auxiliary data $(K, J)$.  However, it can be shown that it induces a well-defined associative product at the level of homology:
  \begin{align*}
   HF(L_0,L_1;p_{01};H_{01}, J_{01}) \otimes HF(L_1,&L_2;p_{12};H_{12}, J_{12}) \\ & \longrightarrow  HF(L_0,L_2;p_{02};H_{02}, J_{02}) \,.      
  \end{align*}

We will refer to this product as the pair-of-pants product. Given Floer homology classes $\alpha, \beta$, we will denote their pair-of-pants product by $\alpha * \beta$.  

\subsubsection{Compatibility of the pair-of-pants product with continuation maps}\label{sec:compat_pop_cont}  
Denote by $H'_{ij}, \,  0\leq i < j \leq 2,$ three  additional Hamiltonians which are non-degenerate with respect to the pairs $(L_i, L_j)$ and pick three almost complex structures $J'_{ij}$ so that the pairs $(H'_{ij},J'_{ij})$ are regular. 
Let $\alpha \in HF(L_0,L_1;p_{01};H_{01},J_{01})$ and $\beta \in HF(L_1,L_2;p_{12};H_{12},J_{12}).$  The pair-of-pants product $*$ is compatible with continuation maps in the following sense:
\begin{equation}\label{eq:compat_cont_pop}
\Psi_{H_{01}, J_{01}}^{H'_{01},J'_{01}}(\alpha) * \Psi_{H_{12},J_{12}}^{H'_{12},J'_{12}}(\beta) = \Psi_{H_{02},J_{02}}^{H'_{02},J'_{02}}(\alpha * \beta) \,.
\end{equation}
One can prove this formula by considering 0-- and 1--dimensional components of suitable moduli spaces of objects combining continuation Floer strips and pair-of-pants strips with slits.

Note that this compatibility between pair-of-pants and continuation maps allows one to consider the pair-of-pants product as a product on Lagrangian Floer homology, independently of the auxiliary data:
  \begin{align*}
   HF(L_0,L_1;p_{01}) \otimes HF(L_1,L_2;p_{12}) \longrightarrow  HF(L_0,L_2;p_{02}) \,.      
  \end{align*}

\subsubsection{The pair-of-pants product when $L_0 =L_1$}\label{sec:pop_intersec_prdct}
As mentioned in Section \ref{sec:when-l_0=l_1}, in the case of a single Lagrangian $L$, the PSS isomorphism intertwines the Morse and Floer theoretic versions of the intersection product. Namely,
\begin{align} \label{eq:pop_intersec_prdct}
\Phi^{L}_{H,J}(a \cdot b)  =  \Phi^{L}_{H,J}(a) * \Phi^{L}_{H,J}(b) 
\end{align}
for any regular pair $(H,J)$ and any two classes $a$ and $b$ in $HM(L)$. So in the case of a single Lagrangian, the pair-of-pants product turns $HF(L,L;p)$ into a ring with unit $\Phi^{L}_{H,J}([L])$, where $[L]$ is the fundamental class of $L$.

\subsubsection{The triangle inequality}\label{subsubsec:triangle_ineq} 
 We continue to work with the Lagrangians $L_0, L_1, L_2$ from the previous sections.  We call a triple $(L_0, L_1, L_2)$ of Lagrangians weakly exact if any disk with boundary on $L_0 \cup L_1 \cup L_2$ and ``corners'' at $p_{01}, p_{12}, p_{02}$ has vanishing symplectic area.  More precisely, denote by $D$ the closed unit disk in $\C$, and fix the three points $z_0 = 1, z_1 = e^{\frac{2\pi}{3}i}, z_2 = e^{-\frac{2\pi}{3}i}$ on the boundary of $D$.  Let $\gamma_0$ denote the segment on the boundary of $D$ between $z_0$ and $z_1$, and similarly define $\gamma_1, \gamma_2$.

\begin{definition}\label{def:weak_exact_triple} 
The triple  $(L_0, L_1, L_2)$ is weakly exact with respect to the intersection points $(p_{01}, p_{12}, p_{02})$ if  $\int_{D}  v^*\omega = 0$ for every disk $$v : (D,\gamma_0, \gamma_1, \gamma_2, z_0, z_1, z_2)  \rightarrow (M, L_0, L_1, L_2, p_{01}, p_{12}, p_{02}).$$
\end{definition}

Our main motivation for introducing the above definition is to establish sharp estimates needed to prove the triangle inequality satisfied by spectral invariants.

\begin{example}\label{exple:exact-triple} 
Weakly exact pairs of Lagrangians in the sense of Definition \ref{def:weak_exact} provide examples of weakly exact triples. Namely, if $(L_0,L_1)$ is weakly exact with respect to $p\in L_0 \cap L_1$, then $(L_0,L_0,L_1)$ and $(L_0,L_1,L_1)$ are weakly exact with respect to $(p,p,p)$ since a disk as in the definition above is a particular case of pinched disks as in Definition \ref{def:weak_exact}. This, combined with Example \ref{ex:weakly_exact}, shows that the triples in which we will be interested in the course of the proof of Theorem \ref{theo:main-theo} (more precisely, in Theorem \ref{theo:triangle_ineq_main} below) are weakly exact with respect to $(p,p,p)$ for any $p\in L_0 \cap L_1$.
\end{example}

Let $H_{01}, H_{12}$ denote any two time-dependent Hamiltonians on $M$.  Define 
\begin{align}\label{eq:defn_sharp}
H_{01} \# H_{12}(t,x) = \left\{ \begin{array}{ll} 
2H_{01}(2t,x) & \text{if} \; t \in [0, \frac{1}{2}] \\ 
 2H_{12}(2t-1,x) & \text{if} \; t \in [\frac{1}{2}, 1]. \end{array} \right.
\end{align}

 Once again, without loss of generality we may assume that both $H_{01}$ and $H_{12}$ vanish for $t$ near $0$ and $1$, see Remark \ref{rem:time_rep}.  Hence, $H_{01} \# H_{12}$ is a smooth Hamiltonian. Observe that $\phi^1_{H_{01} \# H_{12}} = \phi^1_{H_{12}} \circ \phi^1_{H_{01}}.$  
The main goal of this section is to prove the following triangle inequality:  
\begin{theo}\label{theo:triangle_inequality_general}
Let $(L_0,L_1,L_2)$ be a triple of Lagrangians which is weakly exact with respect to $(p_{01}, p_{12}, p_{02}),$ where $p_{ij} \in L_i \cap L_j$.  Denote by $\alpha, \beta$  homology classes in the reference Floer homology groups $HF_{\mathrm{ref}}(L_0,L_1;p_{01})$ and $HF_{\mathrm{ref}}(L_1, L_2;p_{12})$, respectively.  The following inequality holds:
$$\ell(\alpha * \beta ;L_0,L_2;p_{02}; H_{01} \# H_{12}) \leq
  \ell(\alpha;L_0,L_1;p_{01}; H_{01} ) + \ell(\beta;L_1,L_2;p_{12}; H_{12}).$$
\end{theo}  

Note that, the compatibility of the pair-of-pants product with continuation maps, as described in Section \ref{sec:compat_pop_cont}, allows us to view $\alpha * \beta$ in the reference Floer homology group $HF_{\mathrm{ref}}(L_0, L_2;p_{02})$. We will now prove the triangle inequality.

\begin{proof}
  Recall that, by Inequality \eqref{eq:unif_continuity}, the spectral invariant $\ell(\cdot\, ;L_i,L_j;p_{ij}; H)$ depends continuously on $H$. Hence by replacing $H_{01}$, and $H_{12}$ with nearby non-degenerate Hamiltonians if needed, we may assume that $H_{01}, H_{12}$ and $H_{01} \# H_{12}$ are all regular.

 Write $H_{02} = H_{01} \# H_{12}$. As in the previous section, take Hamiltonian chords $x_{ij} \in CF(L_i,L_j;p_{ij};H_{ij})$ and consider the moduli space appearing in the definition of the pair-of-pants product, $\m M(x_{01},x_{12}; x_{02})$.  For any $\epsilon >0$, it is possible to pick the function $K: \R \times [-1,1] \times M \rightarrow \R$ in the auxiliary data $(K, J)$ such that
  \begin{align}\label{eq:s_indep}
  \left| \frac{\partial K_{s,t}} {\partial s} \right| \leq \frac{\epsilon}{4} \;\; \text{if  } s \in [-1, 1], \; \text{and } \frac{\partial K_{s,t}} {\partial s} = 0 \mbox{ otherwise.}
  \end{align} 
  Indeed, this can be achieved by making a small perturbation of the following function
 $$ K'(s,t,x) = \left\{ \begin{array}{ll}
 H_{01}(t+1,x)  & \mbox{ for } t\in [-1, 0 ], \\ 
 H_{12}(t,x) & \mbox{ for  } t \in [0,1]. \end{array}\right. $$
                                                                                                                                                             
We leave it to the reader to verify that proving the triangle inequality amounts to showing that $\m A_{H_{02}}^{L_0,L_2}(x_{02}) \leq \m A_{H_{01}}^{L_0,L_1}(x_{01}) + \m A_{H_{12}}^{L_1,L_2}(x_{12}).$  We will now prove this last inequality.  For any $u \in \m M(x_{01},x_{12}; x_{02}),$ the following holds:	
\begin{align*}
  0 \leq \int_{\Sigma} &\|\partial_su(s,t)\|^2 dsdt = \int_{\Sigma} \omega(\partial_s u, J_{(s,t)}(u) \partial_s u) dsdt \\
 &= \int_{\Sigma} \omega(\partial_s u, \partial_t u - X_{K}^{s,t}(u)) dsdt = \int_{\Sigma} u^* \omega - \int_{\Sigma} dK_{s,t}(\partial_su) dsdt \,.
\end{align*}

Now, let  $\bar{x}_{ij}$ denote  homotopies  from the chords $x_{ij}$ to the constant paths $p_{ij}$, i.e. cappings for $x_{ij}$.  Since, the triple $(L_0, L_1, L_2)$ is weakly exact with respect to $(p_{01}, p_{12}, p_{02})$, the disk $ \bar{x}_{01} \# \bar{x}_{12} \# u \#  (- \bar{x}_{02})$  has symplectic area zero.  Hence, we see that 
$$\int_{\Sigma} u^* \omega =  - \int_{D}  (\bar{x}_{01})^* \omega - \int_{D}  (\bar{x}_{12})^* \omega  + \int_{D}  (\bar{x}_{02})^* \omega \,.$$
On the other hand,  Equation \eqref{eq:s_indep} implies that $\int_{\Sigma} {\partial_s K_{s,t}}(u) \,dsdt \leq \epsilon$ and hence we obtain the following:
\begin{align*}
  - \int_{\Sigma}& dK_{s,t}(\partial_su) dsdt = -  \int_{\Sigma} \partial_s (K_{s,t}\circ u) \,dsdt + \int_{\Sigma} \partial_s K_{s,t}(u) \,dsdt \\
 &\leq  \int_0^1 H_{01}(t, x_{01}(t))dt  + \int_0^1 H_{12}(t, x_{12}(t))dt  - \int_{-1}^1 H_{02}(t, x_{02}(t))dt +  \epsilon \,. 
\end{align*}
We conclude from the above that 
$$0 \leq \int_{\Sigma} \|\partial_su(s,t)\|^2 dsdt \leq  \m A_{H_{01}}^{L_0,L_1}(x_{01}) + \m A_{H_{12}}^{L_1,L_2}(x_{12}) - \m A_{H_{02}}^{L_0,L_2}(x_{02})+ \epsilon$$
which finishes the proof of the triangle inequality.  
\end{proof}

\subsection{A K\"unneth formula for Lagrangian Floer homology and a splitting formula for spectral invariants}
\label{sec:kunn-form-prod}

 Let  $(M', \omega')$ and $(M'', \omega'')$ denote two closed symplectically aspherical symplectic manifolds.  Let $(L_0', L_1')$ denote a pair of Lagrangians in $M'$ which is weakly exact with respect to a fixed intersection point $p' \in L_0' \cap L_1'$.  Take $(H', J')$ to be a regular pair (of a Hamiltonian and an almost complex structure), as defined in Section \ref{sec:Lagr-Floer-theory}, for the weakly exact pair of Lagrangians $(L_0, L_1)$.  Similarly, we define $(L_0'', L_1''),\; p'' \in L_0'' \cap L_1''$, and $(H'', J'')$ in $M''$.


Consider the product Lagrangians $L_0 = L_0' \times L_0'', \; L_1 = L_1' \times L_1''$ in $(M' \times M'', \omega' \oplus \omega'')$, the Hamiltonian $H_t(x,y) = H' \oplus H''(t,(x,y)) := H'_t(x) + H''_t(y)$, and the almost complex structure $J= J' \oplus J''.$  Note that the pair $(L_0, L_1)$ is weakly exact with respect to the intersection point $(p', p'') \in L_0 \cap L_1.$  It is easy to see that the Hamiltonian $H$ is non-degenerate, and moreover, the Hamiltonian chords of $H$ are of the form $x = (x', x'')$ where $x', x''$ are Hamiltonian chords of $H'$ and $H''$. 

The pair $(H, J)$ is regular for $(L_0, L_1)$: This is because the linearization of the operator $u\mapsto \del_s u + J_t(u)(\del_t u - X_H^t(u))$ splits into a product of the corresponding linearizations for $(H', J')$ and $(H'', J'')$; see, for example, \cite{Lec09} for further details.  It follows that, for any two chords $x_- = (x_-', x_-'')$ and $x_+ = (x_+', x_+'')$ of $H$, the moduli space $\m{\widehat{M}}^{L_0,L_1}(x_-,x_+;H,J)$, used in the definition of the Floer boundary map, coincides with the product $$\m{\widehat{M}}^{L_0',L_1'}(x_-',x_+';H',J') \times \m {\widehat{M}}^{L_0'',L_1''}(x_-'',x_+'';H'',J'').$$

We leave it to the reader to conclude from the discussion in the preceding paragraph that 
 $$CF(L_0,L_1;p;H) = CF(L_0',L_1';p';H') \otimes CF(L_0'',L_1'';p'';H'') \,,$$
 where the boundary map $\partial$ is defined by $\partial x = \partial ' x' \otimes x'' + x' \otimes \partial '' x'',$ with $\partial '$ and $\partial ''$ denoting the boundary maps for the Floer complexes of $H'$ and $H''$, respectively. Recall that we are working over $\Z_2$ and thus applying the standard K\"unneth formula we obtain 
 \begin{equation}\label{eq:Kunneth}
 HF(L_0,L_1;p;H, J) = HF(L_0',L_1';p';H', J') \otimes HF(L_0'',L_1'';p'';H'', J'') \,. 
 \end{equation}

\subsubsection{A splitting formula for spectral invariants}\label{sec:prdct_form}
We present in this section a splitting formula\footnote{This is sometimes called the ``product formula'' in the literature.  We have chosen this alternative terminology in order to avoid any possible confusion with the triangle inequality coming from product of homology classes.} for spectral invariants in the situation described above. Consider a Floer homology class $ \alpha = \alpha' \otimes \alpha''\neq 0$ in $HF(L_0',L_1';p';H', J') \otimes HF(L_0'',L_1'';p'';H'', J'')$.  By the discussion above, $\alpha$ is a homology class in $HF(L_0,L_1;p;H, J)$.  The following splitting formula holds:
 \begin{equation}\label{eq:prdct_form}
 \ell(\alpha;L_0,L_1;p; H) = \ell(\alpha';L_0',L_1';p'; H') + \ell(\alpha'';L_0'',L_1'';p''; H'').
 \end{equation}
In \cite[Section 5]{EP09}, a more abstract and general version of the above formula is proven for spectral invariants of ``decorated $\Z_2$--graded complexes'', see \cite[Theorem 5.2]{EP09}.  Formula \eqref{eq:prdct_form} is an immediate corollary of this theorem.

\subsubsection{Compatibility of the K\"unneth Formula with the pair-of-pants product}\label{sec:compat_Kunneth_quant}
 
  In this section, we describe the compatibility of the K\"unneth formula \eqref{eq:Kunneth} with the pair-of-pants product as defined in Section \ref{sec:lagr-pdct-triangle-ineq}. 
  
   Let $L_0 = L_0' \times L_0'', \, L_1 = L_1' \times L_1'' \subset M' \times M''$  be as in the previous section, and consider additionally a third Lagrangian $L_2 = L_2' \times L_2''.$  For $0 \leq  i < j \leq 2,$ take three intersection points  $p_{ij} = (p_{ij}', p_{ij}'') \in L_i \cap L_j$ and suppose that $(L_i', L_j'), \, (L_i'', L_j'')$ are weakly exact with respect to the intersection points $p_{ij}', \, p_{ij}'',$ respectively.  Lastly, let $(H_{ij}', J_{ij}')$ and  $(H_{ij}'', J_{ij}'')$ denote regular pairs for $(L_i', L_j')$ and $(L_i'', L_j'')$, respectively.  
  
  As in the previous section, we consider the split Hamiltonians and almost complex structures $H_{ij} = H_{ij}' \oplus H_{ij}'', \, J_{ij} = J_{ij}' \oplus J_{ij}''.$  
 By the K\"unneth formula \eqref{eq:Kunneth}, $HF(L_i, L_j; p_{ij}; H_{ij}, J_{ij})$ is generated by elements of the form $\alpha' \otimes \alpha''$, where $\alpha' \in HF(L_i', L_j';p_{ij}'; H_{ij}', J_{ij}')$  and $\alpha'' \in HF(L_i'', L_j'';p_{ij}''; H_{ij}'', J_{ij}'').$ Therefore, describing the pair-of-pants product, in this setting, reduces to describing the product for such elements.
 
  Consider $ \alpha' \otimes \alpha'' \in HF(L_0, L_1;p_{01}; H_{01}, J_{01})$ and $\beta' \otimes \beta'' \in  HF(L_1, L_2;p_{12} \\ ; H_{12}, J_{12}).$ Then, the following equality holds  in $HF(L_0, L_2;p_{02}; H_{02}, J_{02}):$
   \begin{equation}\label{eq:pop_prdct_Kunn}
(\alpha' \otimes \alpha'') * (\beta' \otimes \beta'')= (\alpha' * \beta') \otimes (\alpha'' * \beta'').
 \end{equation}
 The reasoning as to why the above holds is very similar to the reasoning for the K\"unneth formula \eqref{eq:Kunneth}: Let $x_{ij} = (x_{ij}', x_{ij}'')$ denote Hamiltonian chords for $H_{ij} = H_{ij}' \oplus H_{ij}''$.  Recall from Section \ref{sec:lagr-pdct-triangle-ineq} the moduli space $\m M(x_{01},x_{12}; x_{02})$ which is used to define the pair-of-pants product.  Such moduli spaces  split into products: $$\m M(x_{01},x_{12}; x_{02}) = \m M(x_{01}',x_{12}'; x_{02}') \times \m M(x_{01}'',x_{12}''; x_{02}'').$$

 \subsubsection{Compatibility of the K\"unneth formula with the PSS isomorphism and the splitting formula.}\label{sec:compat_Kun_PSS}
Consider the case of a single Lagrangian $L = L' \times L''$. The PSS morphism as described in this particular case in Section \ref{sec:when-l_0=l_1} is compatible with the K\"unneth formula \eqref{eq:Kunneth}. This was the content of \cite[Claim 3.4]{Lec09} in the more general case of monotone manifolds. More precisely, the Morse theoretic version of K\"unneth's formula is satisfied, that is
\begin{align*}
  HM(L) = HM(L') \otimes HM(L'') \,.
\end{align*}
As in the Floer theoretic case this can be proven, even at the chain level, by choosing a Morse--Smale pair $(f,g)$ on $L$ which splits, that is $f=f'\oplus f''$ and $g=g'\oplus g''$ where $(f',g')$ and $(f'',g'')$ are Morse--Smale pairs for $L'$ and $L''$, respectively.

Now, for such split Morse and Floer data, respectively $(f,g)$ and $(H,J)$, one can easily prove that for all $a' \in HM(L';f',g')$ and $a'' \in HM(L'';f'',g'')$, 
 \begin{align}\label{eq:compat_Kun_PSS}
 \Phi^{L}_{H,J}(a' \otimes a'') = \Phi^{L'}_{H',J'}(a') \otimes  \Phi^{ L''}_{H'',J''}(a'')
 \end{align}
again, even at the chain level, since the moduli spaces involved in the construction of the PSS isomorphism themselves split along the product. (The fact that the PSS isomorphism is compatible with Morse and Floer continuation morphisms then allows one to consider, at the homological level, non necessarily split data.)

It follows that the splitting formula \eqref{eq:prdct_form} restricts to the following:
\begin{align} \label{eq:prdct_form_l0=l1}
 \ell(a \otimes b; L,L; H' \oplus H'') = \ell(a; L', L'; H') + \ell(b; L'', L''; H'')
\end{align}
for all non-zero Morse homology classes $a \in HM(L')$ and $b \in HM(L'')$. Notice that this corresponds to \cite[Theorem 2.14]{MVZ12} in the case of Hamiltonians with complete flows on cotangent bundles.

%
%

\section{Lagrangian Floer theory of tori}
\label{sec:lag_Floer_tori}
In this section, we specialize the theory developed in Section \ref{sec:preliminaries-Floer} to the settings introduced in Section \ref{sec:intro_tech}.  Recall that, $M = \bb T^{2k_1} \times \bb T^{2k_1} \times \bb T^{2k_2} \times \bb T^{2k_2}$ and that it is equipped with the standard symplectic form $\omega_{\mathrm{std}}.$
 Furthermore, recall that the Lagrangians we are interested in, $L_0$ and $L_1$, are defined as follows
\begin{align*}
  L_0 &= \bb T^{k_1} \times \{ 0 \} \times \bb T^{k_1} \times \{ 0 \} \times \bb T^{k_2} \times \{ 0 \} \times \bb T^{k_2} \times \{ 0 \} \,, \\
  L_1 &= \bb T^{k_1} \times \{ 0 \} \times \bb T^{k_1} \times \{ 0 \} \times \bb T^{k_2} \times \{ 0 \} \times \{ 0 \} \times \bb T^{k_2} \,.
\end{align*}
As noted in Section \ref{sec:intro_tech}, $L_i = L \times L_i'$, where 
\begin{align*}
  L = &\;\bb T^{k_1} \times \{ 0 \} \times \bb T^{k_1} \times \{ 0 \} \times \bb T^{k_2} \times \{ 0 \}  \;\subset\; \bb T^{2k_1} \times \bb T^{2k_1} \times \bb T^{2k_2} \,, \\
  &L_0' =  \bb T^{k_2} \times \{ 0 \} \;\subset\ \bb T^{2k_2},  \mbox{ and }\; L_1' = \{ 0 \} \times \bb T^{k_2} \;\subset\; \bb T^{2k_2} \,.
\end{align*}

The pairs of Lagrangians $(L,L)$, $(L_0', L_1')$, and $(L_0, L_1) = (L \times L_0', L \times L_1')$ are all weakly exact with respect to any point in their corresponding intersections; see Example \ref{ex:weakly_exact}.  We fix, for the rest of this article,  in the intersection of each of the above pairs the point all of whose coordinates are zero, and carry out the constructions of Floer homology and spectral invariants (as described in Section \ref{sec:preliminaries-Floer}) with respect to this intersection point. We will omit the intersection point from our notation.

\subsection{$HF(L_0, L_1)$ and the associated spectral invariants} \label{sec:spec-floer-theory}  
In this section, we construct an isomorphism between Morse homology $HM(L)$ and Floer homology $HF(L_0,L_1;H, J)$.  We will then use this isomorphism to associate spectral invariants to Morse homology classes.  

\begin{theo}\label{theo:Floer-hom-L0L1}
 There exists a PSS-type isomorphism 
$$\Phi^{L_0, L_1}_{H,J} : HM(L) \rightarrow HF(L_0, L_1; H, J),$$
associated to every regular pair $(H, J)$. 
Furthermore, the isomorphism $\Phi^{L_0, L_1}_{H,J}$ is compatible with continuation morphisms in the following sense:
\begin{align}\label{eq:PSS-comm-continuation2}
\Phi^{L_0, L_1}_{H,J} = \Psi_{H', J'}^{H, J} \circ \Phi^{L_0, L_1}_{H',J'},
\end{align}
where $(H', J')$ is any other regular pair.
\end{theo}
\begin{proof} 
Pick a Hamiltonian $F$ and an almost complex structure $j$, on $\bb T^{2k_1} \times \bb T^{2k_1} \times \bb T^{2k_2}$, such that $(F, j)$ is regular for the pair of Lagrangians $(L,L)$. Similarly, we pick $(F', j')$, on $\bb T^{2k_2}$, such that $(F', j')$ is regular for the pair $(L_0', L_1')$.  Define the Hamiltonian $F \oplus F'$ on $M$ by $F \oplus F' (z_1, z_2)= F(z_1) + F'(z_2),$ for $z_1 \in \bb T^{2k_1} \times \bb T^{2k_1} \times \bb T^{2k_2}$ and $z_2 \in  \bb T^{2k_2}.$

We know from Sections \ref{sec:when-l_0=l_1} and \ref{sec:when-l_0cap-l_1} that there exist a PSS isomorphism $\Phi^{L}_{F,j} : HM(L) \rightarrow HF(L,L; F, j)$ and a PSS-type isomorphism between $\Phi^{L_0',L_1'}_{F',j'} :\Z_2 \rightarrow HF(L_0', L_1'; F', j').$  

We define $$\Phi^{L_0, L_1}_{F\oplus F',j \oplus j'} : HM(L)\rightarrow HF(L,L; F, j) \otimes HF(L_0',L_1'; F', j')$$ to be the tensor product of these two isomorphisms, i.e. $\forall \, a \in HM(L)$ we have:
\begin{align}\label{eq:PSS_general}
\Phi^{L_0, L_1}_{F\oplus F', j \oplus j'} (a) = \Phi^{L}_{F,j}(a) \otimes [\mathrm{pt}],
\end{align} 
where $[\mathrm{pt}]$ denotes the non-trivial homology class in $HF(L_0', L_1'; F', j').$
On the other hand, the K\"unneth formula \eqref{eq:Kunneth} tells us that 
$$HF(L_0,L_1; F\oplus F', j \oplus j') =  HF(L,L; F, j) \otimes HF(L_0',L_1'; F', j').$$
Hence, $\Phi^{L_0, L_1}_{F\oplus F', j \oplus j'}$ gives an isomorphism between $HM(L)$ and $HF(L_0,L_1; F\oplus F', j \oplus j').$  For an arbitrary regular pair $(H, J)$ we define $\Phi^{L_0, L_1}_{H,J} : HM(L) \rightarrow \, HF(L_0, L_1; H, J)$ by the formula $$\Phi^{L_0, L_1}_{H,J} = \Psi_{F\oplus F', j \oplus j'}^{H,J}\circ \Phi^{L_0, L_1}_{F\oplus F', j \oplus j'}.$$

The proof of the second half of the theorem, concerning the compatibility of $\Phi^{L_0, L_1}_{H,J}$ with continuation morphisms, is an immediate consequence of 
the definition above and the fact that continuation morphisms are canonical in the sense of Equation \eqref{eq:compos_cont_maps}.
\end{proof}

As an immediate consequence of Theorem \ref{theo:Floer-hom-L0L1}, we can associate spectral invariants to elements of $HM(L)$. 

\begin{definition}\label{def:specialized_lag_spec_inv}  For every regular Hamiltonian  $H$, the spectral invariant associated to a non-zero element $a \in HM(L)$ is the number
$$\ell(a;L_0,L_1;H):=\ell(\Phi^{L_0, L_1}_{H,J}(a);L_0,L_1;H),$$
where the right-hand side  is defined via Equation \eqref{def:lag_spec_inv_1}.
\end{definition}

\begin{remark}\label{rem:definitions_are_compatible}
The above definition is compatible with Definition \ref{def:lag_spec_inv}.  In effect, what we have done in the above definition is to take the reference pair $(\Href, \Jref)$ of Definition \ref{def:lag_spec_inv} to be the pair $(F\oplus F', j \oplus j')$ from the proof of Theorem \ref{theo:Floer-hom-L0L1}.
\end{remark}

\medskip

It follows immediately from  Inequality \eqref{eq:unif_continuity} 
that for non-degenerate $H$, $H'$, we have: 
\begin{align}\label{eq:unif_continuity_split} \begin{split}
   \int_0^1 \min_M (H_t-H'_t) dt \leq |\ell(a;L_0,L_1;H)-&\ell(a;L_0,L_1;H')| \\
&\qquad \leq \int_0^1 \max_M (H_t-H'_t) dt \,. \end{split}
\end{align}
  We conclude that $\ell(a;L_0,L_1;H)$ can be defined by continuity for every continuous function $H:[0,1]\times M\to\R$.
  
\medskip 

We end this subsection by mentioning that, as in Section \ref{sec:spectrality}, one can easily verify the spectrality property, $\ell(a; L_0, L_1; H) \in  \mathrm{Spec}(H).$

\subsection{Product structure and the triangle inequality}\label{sec:prdct_triangle_l0l1}
Recall that the  Morse homology of any manifold $X$ carries a ring structure where the product of $a, b \in HM(X)$ is given by the intersection product $a \cdot b$.  

Consider the Lagrangian submanifolds $L_0 = L \times L_0', \; L_1 = L \times L_1'.$  As a consequence of the K\"unneth formula for Morse homology, the homology ring $HM(L_i)$ can be written as the tensor product of the rings $HM(L)$ and $HM(L_i')$, i.e. $$HM(L_i) = HM(L) \otimes HM(L_i').$$
For $1 \leq j \leq 3$,  let $(H_j, J_j)$ denote three regular pairs. Recall that in Section \ref{sec:lagr-pdct-triangle-ineq} we defined the pair-of-pants product.  Here, we will consider the following two instances of the pair-of-pants product 
\begin{align*}
*: HF(L_0,L_1;H_1, J_1) \otimes HF(L_1,L_1; H_2, J_2) \rightarrow  HF(L_0,L_1;H_3, J_3), \\
*: HF(L_0,L_0;H_1, J_1) \otimes HF(L_0,L_1; H_2, J_2) \rightarrow  HF(L_0,L_1;H_3, J_3).
\end{align*}  
The next theorem describes the relation between the above two pair-of-pants products and the intersection product on $HM(L)$.

\begin{theo}\label{theo: prdct_struct_L0L1}
 Denote by $[L_i'],$ for $i=0,1,$ the fundamental class in $HM(L_i')$ and by $a, b \in HM(L)$ any two Morse homology classes.  The intersection and pair-of-pants products satisfy the following relations:
\begin{enumerate}
\item $\Phi_{H_1, J_1}^{L_0, L_1} (a) * \Phi_{H_2, J_2}^{L_1}(b \otimes [L_1'])=  \Phi_{H_3, J_3}^{L_0, L_1} (a\cdot b),$
\item $\Phi_{H_1, J_1}^{L_0} (a \otimes [L_0']) * \Phi_{H_2, J_2}^{L_0, L_1}(b)=  \Phi_{H_3, J_3}^{L_0, L_1} (a\cdot b).$
\end{enumerate}
\end{theo}
\begin{proof}
We will only prove the first of the above two identities.  The second is proven in a similar fashion.

Recall that $L \subset \bb T^{2k_1} \times \bb T^{2k_1} \times \bb T^{2k_2}$ and  $L_0', L_1' \subset \bb T^{2k_2}$.  Let $F$ and $F'$ denote two non-degenerate Hamiltonians on $\bb T^{2k_1} \times \bb T^{2k_1} \times \bb T^{2k_2}$ and $\bb T^{2k_2}$, respectively.   We claim that it is sufficient to prove the theorem in the special case where $H_1 = H_2 = H_3 = F \oplus F'.$ Indeed, this can be deduced using the following two ingredients: First, the compatibility of PSS  and PSS-type isomorphisms  with  continuation isomorphisms, as described by Diagram \eqref{eq:PSS-comm-continuation} and Equation \eqref{eq:PSS-comm-continuation2}.  Second, the compatibility of continuation isomorphisms with the pair-of-pants product as described by Equation \eqref{eq:compat_cont_pop}. We will prove the theorem in this special case and leave it to the reader to verify that this indeed does imply the general case.

  Pick almost complex structures $J, J'$ such that the pairs $(F,J)$ and $(F', J')$ are both regular.  Now, it follows from Equation \eqref{eq:PSS_general} that $$\Phi^{L_0, L_1}_{F \oplus F',J \oplus J'}(a) = \Phi^{L}_{F,J}(a) \otimes [\mathrm{pt}],$$
where $[\mathrm{pt}]$ denotes the non-trivial homology class in $HF(L_0', L_1'; F', j').$
We also know, from Equation \eqref{eq:compat_Kun_PSS}, that $$\Phi^{L_1}_{F\oplus F',J\oplus J'}(b \otimes [L_1']) = \Phi^{L}_{F,J}(b) \otimes  \Phi^{L_1'}_{F',J'}([L_1']).$$
We will be needing the following identity, whose proof we postpone for the time being:
\begin{equation}\label{eq:prdct}
[\mathrm{pt}] * \Phi^{L_1'}_{F',J'}([L_1']) =  [\mathrm{pt}].
\end{equation}

Using the above, we obtain the following:
\begin{align*}
 \Phi&^{L_0, L_1}_{F \oplus F',J \oplus J'}(a) * \Phi^{L_1}_{F\oplus F',J\oplus J'}(b \otimes [L_1']) & \\
&=  \left(\Phi^{L}_{F,J}(a) \otimes [\mathrm{pt}]\right) * \left(\Phi^{L}_{F,J}(b) \otimes  \Phi^{L_1'}_{F',J'}([L_1'])\right) & \\
&= \left(\Phi^{L}_{F,J}(a) * \Phi^{L}_{F,J}(b)\right) \otimes \left([\mathrm{pt}] * \Phi^{ L_1'}_{F',J'}([L_1'])\right) & \text{by Equation } \eqref{eq:pop_prdct_Kunn}\;\\
&= \Phi^{L}_{F,J}(a\cdot b) \otimes \left([\mathrm{pt}] * \Phi^{L_1'}_{F',J'}([L_1'])\right) & \text{by Equation } \eqref{eq:pop_intersec_prdct}\;\\
&= \Phi^{L}_{F,J}(a\cdot b) \otimes [\mathrm{pt}]  & \text{by Equation } \eqref{eq:prdct}\;\\
&= \Phi^{L_0, L_1}_{F \oplus F',J \oplus J'}(a\cdot b) & \text{by Equation } \eqref{eq:PSS_general}.
\end{align*}

\medskip
It remains to prove Equation \eqref{eq:prdct}. Recall that  $$
  L_0' =  \bb T^{k_2} \times \{ 0 \} \subset \bb T^{2k_2}, \, 
  L_1' = \{ 0 \} \times \bb T^{k_2} \subset \bb T^{2k_2} \,.
$$
 Let $$
\Lambda_0 = \bb T^{1} \times \{ 0 \} \subset \bb T^{2}, \, 
  \Lambda_1 = \{ 0 \} \times \bb T^{1} \subset \bb T^{2} \, ,$$
 and observe that (up to a symplectomorphism) $$L_0' = \underbrace{ \Lambda_0 \times \cdots \times  \Lambda_0}_{k_2 \text{ times}}, \,
 L_1' = \underbrace{ \Lambda_1 \times \cdots \times \Lambda_1}_{k_2 \text{ times}}.$$

 Equations \eqref{eq:PSS-comm-continuation} and \eqref{eq:compat_cont_pop} tell us, respectively, that continuation morphisms are compatible with the PSS isomorphism and the pair-of-pants product.  From this, one can conclude that it is sufficient to verify Equation \eqref{eq:prdct} for any specific choice of a regular pair $(F',J')$. Furthermore, the regular pair used for defining $HF(L_0', L_1')$ can indeed be different from the one used for defining $HF(L_1', L_1')$.  We will verify the formula for the choices described in the next two paragraphs.
  
   For $HF(L_1', L_1')$ we pick a regular pair of the form
 $$(\underbrace{ f\oplus \cdots \oplus f}_{k_2 \text{ times}}, \, \underbrace{j \oplus \cdots \oplus j)}_{k_2 \text{ times}},$$
 where $(f, j)$ denotes a regular pair on the $2$--torus $\bb T^{2},$ the almost complex structure $j \oplus \ldots \oplus j$ denotes the obvious split almost complex structure on $\bb T^{k_2}$, and  $f\oplus \ldots \oplus f(z_1, \ldots, z_{k_2}) = f(z_1)+ \ldots + f(z_{k_2}).$ 
 
    For $HF(L_0', L_1')$ we pick a regular pair of the form 
$$(0, j_0 \oplus \ldots \oplus j_0),$$ 
where $0$ denotes the zero Hamiltonian.  Recall that since $L_0'$ and $L_1'$ intersect transversely the zero Hamiltonian is non-degenerate for this pair.

Now, by the K\"unneth formula \eqref{eq:Kunneth} we have the following splittings 
\begin{align*}
  &HF(L_0', L_1'; 0, j_0 \oplus \cdots \oplus j_0) = HF(\Lambda_0, \Lambda_1; 0,j_0)^{\otimes k_2} , \mbox{ and }\\
  &HF(L_1', L_1'; f\oplus \cdots \oplus f, j \oplus \cdots \oplus j ) = HF(\Lambda_1, \Lambda_1; f,j) ^{\otimes k_2} .
\end{align*}
Furthermore, the above splittings are compatible with the pair-of-pants product as described by Equation \eqref{eq:pop_prdct_Kunn}.  This implies that Equation \eqref{eq:prdct} is an immediate consequence of the following claim:

 \begin{claim*}  Denote by $p$ the unique intersection point of $\Lambda_0$ and $\Lambda_1$ and by $[\mathrm{p}]$ the Floer homology class represented by this point.  Note that this is the unique non-zero class in $HF(\Lambda_0, \Lambda_1; 0,j_0)$. The pair-of-pants product $*: HF(\Lambda_0, \Lambda_1;0,j_0) \otimes HF(\Lambda_1, \Lambda_1; f,j) \rightarrow HF(\Lambda_0, \Lambda_1,;0, j_0)$ satisfies the following identity:
\begin{align}\label{eq:prdct_basecase}
[\mathrm{p}] * \Phi^{\Lambda_1}_{f,j}([\Lambda_1]) = [\mathrm{p}].
\end{align}
\end{claim*}
The proof of the above claim boils down to computing one of the simplest instances of the pair-of-pants product.  This is well-known to experts, and thus we only present a sketch of the proof while avoiding the technical details.

\begin{proof}[Proof of Claim] Once again, by Equations \eqref{eq:PSS-comm-continuation} and \eqref{eq:compat_cont_pop}, it is sufficient to verify the above claim for a given choice of a regular pair $(f, j)$.  We pick $f$ as follows:  begin with a $C^2$--small Morse function  on the circle $\Lambda_1$ and extend it trivially to the product $\bb T^2 = \Lambda_0 \times \Lambda_1$.  Furthermore, we pick $f$ such that $f|_{\Lambda_1}$ has  only two critical points.  Denote the maximum by $Q$ and the minimum by $q$. 

 Let $\Lambda_1' = \phi^1_f(\Lambda_1)$.  Since $f$ is $C^2$--small, $\Lambda_1'$ intersects $\Lambda_0$ transversely at a single point.  We denote this point by $p'$.  See Figure \ref{fig:pop-in-T2}.
 
\begin{figure}[h]
  \centering
  \includegraphics[width=7.0cm]{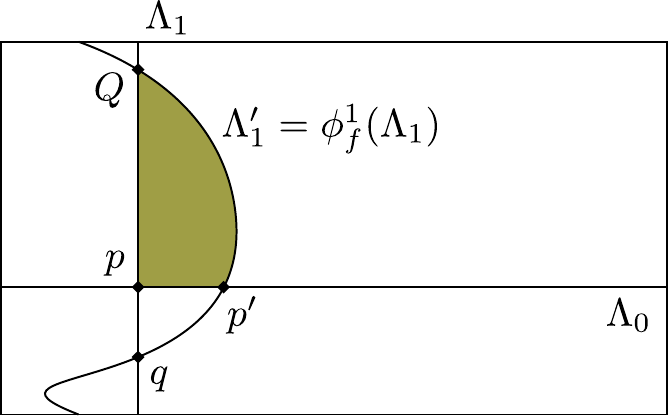}
  \caption{Pair-of-pants product in $\bb T^2$}
  \label{fig:pop-in-T2}
\end{figure}

 It is well-known that there exists a natural identification of $HF(\Lambda_1, \Lambda_1; f, j)$  with $HF(\Lambda_1, \Lambda_1'; 0, \tilde{j})$, where $\tilde{j}_t = (\phi^t_f)^{-1}_* j_t$; see for example \cite[Section 2.2.2]{Leclercq08} or \cite[Remark 1.10]{Au13}.
 The point $Q$ represents a homology class $[\mathrm{Q}] \in HF(\Lambda_1, \Lambda_1'; 0, \tilde{j})$, which corresponds to the fundamental class $[\Lambda_1] \in HM(\Lambda_1).$  Also, $HF(\Lambda_0, \Lambda_1';0, j_0)$ is generated by $[\mathrm{p'}]$, the homology class represented by the point $p'$.
 
 Again, using Diagram \eqref{eq:PSS-comm-continuation} and Equation \eqref{eq:compat_cont_pop}, we see that Equation \eqref{eq:prdct_basecase} is equivalent to 
 \begin{equation} \label{eq:prdct_basecase2}
 [\mathrm{p}] * [\mathrm{Q}] = [\mathrm{p'}]. 
 \end{equation}
  This last equality can be verified without much difficulty.  First, note that $[\mathrm{p}] \in HF(\Lambda_0, \Lambda_1; 0, j_0), \; [\mathrm{Q}] \in HF(\Lambda_1, \Lambda_1'; 0, \tilde{j})$, and $[\mathrm{p'}] \in HF(\Lambda_0, \Lambda_1; 0, j_0),$  and thus, the Hamiltonians in question are all zero.  Furthermore, we can take $j_0$ to be the standard complex structure on $\bb T^2$ and $j$ such that $\tilde{j} = j_0$.  Hence, to verify that $[\mathrm{p}] * [\mathrm{Q}] = [\mathrm{p'}],$  we must count the number of holomorphic disks on $\bb T^2$ with boundary on $\Lambda_0 \cup \Lambda_1 \cup \Lambda_0'$ and corners at the points $p$, $Q$, and $p'$.  We leave it to the reader to verify that there exists only one such disk: the one highlighted in Figure \ref{fig:pop-in-T2}.  This proves Equation \eqref{eq:prdct_basecase2}.
\end{proof}
This completes the proof of Theorem \ref{theo: prdct_struct_L0L1}.
\end{proof}

\begin{remark}\label{rem:full_desc_prdct}
For $i=0,1,$ denote by $H_k(L_i')$ the Morse homology group of degree $k$ of $L_i'$. Suppose that $x \in H_k(L_i')$ where $k\neq \dim(L_i')$.  Then, one can modify the proof of Theorem \ref{theo: prdct_struct_L0L1} to obtain the following additional identities:
\begin{enumerate}
\item $\Phi_{H_1, J_1}^{L_0, L_1} (a) * \Phi_{H_2, J_2}^{L_1}(b \otimes x)= 0,$
\item $\Phi_{H_1, J_1}^{L_0} (a \otimes x) * \Phi_{H_2, J_2}^{L_0, L_1}(b)= 0.$
\end{enumerate}
The above combined with Theorem \ref{theo: prdct_struct_L0L1}, give us a full description of the relation between the intersection product on $HM(L)$ and the pair-of-pants products
\begin{align*}
* \co HF(L_0, L_1) \otimes HF(L_1, L_1) \rightarrow  HF(L_0, L_1),\\
* \co HF(L_0, L_0) \otimes HF(L_0, L_1) \rightarrow  HF(L_0, L_1).  
\end{align*}
We will not prove these additional identities, as they are not needed for the proof of Theorem \ref{theo:main-theo} and their proofs are very similar to the proof of the previous theorem.  We mention here that, in the same way that the proof of Theorem \ref{theo: prdct_struct_L0L1} was reduced to establishing Equation \eqref{eq:prdct_basecase2}, proving these identities reduces  to showing the following:
 \begin{align*}
 [\mathrm{p}] * [\mathrm{q}] = 0,
\end{align*} 
where $p$ and $q$ are defined as in Figure \ref{fig:pop-in-T2}.
\end{remark}

\subsubsection{The triangle inequality} \label{sec:traingle_ineq_l0l1}
In this section, we use Theorem \ref{theo: prdct_struct_L0L1} to prove the two triangle inequalities mentioned in the introduction.  Recall that given two Hamiltonians $H, H'$, their concatenation $H \# H'$ is defined by 
\begin{align*}
H \# H' (t,x) = \left\{ \begin{array}{ll} 
2H(2t,x) & \text{if} \; t \in [0, \frac{1}{2}], \\ 
2H'(2t-1,x) & \text{if} \; t \in [\frac{1}{2}, 1]. \end{array} \right.
\end{align*}

\begin{theo}\label{theo:triangle_ineq_main}
Denote by $[L_i'],$ for $i=0,1,$  the fundamental class in $HM(L_i')$ and by $a, b \in HM(L)$ any two Morse homology classes such that $a\cdot b \neq 0$.  The following inequalities hold:
\begin{enumerate}
\item $\ell(a\cdot b; L_0, L_1; H \# H') \leq  \ell(a ; L_0, L_1; H)+ \ell(b \otimes [L_1']; L_1, L_1;  H'),$
\item $\ell(a\cdot b; L_0, L_1; H \# H') \leq  \ell(a \otimes [L_0']; L_0, L_0 ;H)  + \ell(b; L_0, L_1;  H').$
\end{enumerate}
\end{theo}
\begin{proof}  We will only prove the first of the two inequalities, as the second one is proven in a very similar fashion.

By continuity of spectral invariants \eqref{eq:unif_continuity_split}, it is sufficient to prove the inequality in the special case where $H$, $H'$ and $H\#H'$  are all non-degenerate.  Pick an almost complex structure $J$ such that the pairs $(H, J), (H', J),$ and $(H\#H', J)$ are all regular.  Now, the triangle inequality becomes  a simple consequence of Theorem \ref{theo: prdct_struct_L0L1} and the triangle inequality of Theorem \ref{theo:triangle_inequality_general}.  Indeed,
\begin{align*}
\ell(a\cdot b; L_0, L_1; &H \# H') = \ell ( \Phi_{H\#H', J}^{L_0, L_1} (a\cdot b); L_0, L_1; H \# H')\\
& = \ell(\Phi_{H, J}^{L_0, L_1} (a) * \Phi_{H', J}^{L_1}(b \otimes [L_1']); L_0, L_1; H\#H')\\ 
& \leq \ell(\Phi_{H, J}^{L_0, L_1} (a); L_0, L_1; H) + \ell(\Phi_{H', J'}^{ L_1} (b \otimes [L_1']); L_1, L_1; H') \\
& = \ell(a; L_0, L_1; H) + \ell(b\otimes [L_1']; L_1, L_1; H').
\end{align*}
Note that the triangle inequality of Theorem \ref{theo:triangle_inequality_general} can be applied here because by Example \ref{exple:exact-triple} since $(L_0,L_1)$ is a weakly exact pair,  $(L_0, L_1, L_1)$ is a weakly exact triple.
\end{proof}
 
We now use the triangle inequality to prove Proposition \ref{cor:triangle-inequality-1}.

\begin{proof}[Proof of Proposition \ref{cor:triangle-inequality-1}]
We will only provide a proof in the case $i=1.$  The other case is proven in a similar fashion.

Applying Theorem \ref{theo:triangle_ineq_main} in the special case where $b= [L]$ and  $H = 0$ we obtain $$\ell(a; L_0, L_1; 0 \# H') \leq \ell(a; L_0, L_1; 0) + \ell([L_1]; L_1, L_1;  H').$$ Now, $\ell(a; L_0, L_1; 0)= 0$ because the spectrum of the zero Hamiltonian is the singleton $\{0\}$.  The  Hamiltonian $0 \# H'$ is a reparametrization of $H'$ and so, by Remark \ref{rem:time_rep}, $\ell(a; L_0, L_1; 0 \# H') = \ell(a; L_0, L_1;  H').$ This finishes the proof. 
\end{proof}

\subsection{A splitting formula} \label{sec:prdct_formula_l0l1}  In this subsection, we will use the splitting formula (\ref{eq:prdct_form}) to obtain a similar formula in our current setting.  This will be used in the proof of Theorem \ref{sec:proof}.

Recall that, $M = \bb T^{2k_1} \times \bb T^{2k_1} \times \bb T^{2k_2} \times \bb T^{2k_2}.$ 
Let $F$ and $F'$ denote two Hamiltonians on $\bb T^{2k_1} \times \bb T^{2k_1} \times \bb T^{2k_2}$ and $\bb T^{2k_2}$, respectively.  Define the Hamiltonian $F \oplus F'$ on $M$ by $F \oplus F' (z_1, z_2)= F(z_1) + F'(z_2),$ for $z_1 \in \bb T^{2k_1} \times \bb T^{2k_1} \times \bb T^{2k_2}$ and $z_2 \in  \bb T^{2k_2}.$  This Hamiltonian appeared in the proof of Theorem \ref{theo:Floer-hom-L0L1}. 

\begin{theo}\label{theo:prdct_form_l0l1}
Let $F\oplus F'$ denote any Hamiltonian of the form described in the previous paragraph.  Let $a \in HM(L)$ denote a non-zero class.  The following formula holds:  
$$\ell(a; L_0, L_1; F \oplus F') = \ell(a; L, L ; F) + \ell([\mathrm{pt}]; L_0', L_1'; F'),$$
where $[\mathrm{pt}]$ denotes the non-zero class in $HF(L_0', L_1')$.
\end{theo}
\begin{proof}
By continuity of spectral invariants \eqref{eq:unif_continuity_split}, it is sufficient to prove the theorem for non-degenerate $F$ and $F'$.  Pick almost complex structures $J$ and $J'$ such that the pairs $(F, J)$ and $(F', J')$ are regular. We have the following chain of equalities:
\begin{align*}
\ell(a ;L_0,  &L_1;F \oplus F') &\\ 
&= \ell(\Phi_{F \oplus F', J\oplus J'}^{L_0,L_1}(a);L_0,L_1;F \oplus F') &\mbox{by Definition } \ref{def:specialized_lag_spec_inv}\;\;\,\\
 &=  \ell(\Phi_{F, J}^{L}(a)\otimes\ [\mathrm{pt}];L_0,L_1;F) &\mbox{by Equation }\eqref{eq:PSS_general} \;\\
&= \ell(\Phi_{F_, J}^{L}(a);L,L;F)+\ell([\mathrm{pt}]; L_0',L_1'; F') &\mbox{by Equation }\eqref{eq:prdct_form} \; \\
&= \ell(a;L,L;F)+\ell([\mathrm{pt}]; L_0',L_1'; F') &\mbox{by Definition }\ref{eq:SI-L0=L1}. \;\,
\end{align*} 
\end{proof}


\section{Proof of the main theorem (Theorem \ref{theo:main-theo})} \label{sec:proof}
We start this section by introducing the notations needed for the proof. 
As in Theorem \ref{theo:main-theo}, we consider a symplectic homeomorphism $\phi$ of $(\bb T^{2k_1}\times\bb T^{2k_2}, \omega_{\mathrm{std}})$ which preserves the coisotropic submanifold $C=\bb T^{2k_1}\times\bb T^{k_2}\times \{0\}^{k_2}$.  Observe that the characteristic foliation $\m F$ is parallel to the subtorus $\{0\}^{2k_1}\times\bb T^{k_2}\times \{0\}^{k_2}$. The map $\phi$ induces a homeomorphism $\phi_R$ on the reduced space $\m R=C/\m F=\bb T^{2k_1}$. Throughout the proof, given a homeomorphism $\theta$ between two spaces $X$ and $Y$, and a time-dependent function $\rho$ on $Y$, i.e. a function $\rho:[0,1]\times Y\to\R$, the composition $\rho\circ\theta$ will denote, with a slight abuse of notation, the time dependent function on $X$ defined by $\rho\circ\theta(t,x)=\rho(t,\theta(x))$ for all $t\in[0,1]$ and $x\in X$.
We want to show that $\phi_R$ preserves the spectral invariant $c_+$, i.e. for every time-dependent continuous function $f_R$ on $\m R$, $c_+(f_R\circ \phi_R)=c_+(f_R)$. 

Let $f_R$ be a time-dependent continuous function on the reduced space $\m R$ and denote $g_R=f_R\circ \phi_R$.
We denote by $f$ and $g$, respectively, the standard lifts of $f_R$ and $g_R$ to $\bb T^{2k_1}\times\bb T^{2k_2}$, given by  $f(z_1,z_2)=f_R(z_1)$ and $g(z_1,z_2)=g_R(z_1)$, for all $z_1\in \bb T^{2k_1}$, $z_2\in \bb T^{2k_2}$.
Note that by construction, $f$ coincides with $g\circ\phi^{-1}$ on the coisotropic submanifold $C$.
The situation is summarized in the following diagram:
 \begin{align*}
  \xymatrix@R=.4cm@C=.6cm{\relax
    \bb T^{2k_1+2k_2} \ar[r]^{\phi} & \bb T^{2k_1+2k_2}  & g \ar[rr] \ar[d] && g\circ \phi^{-1} \ar@{}[r]|{\hspace{.4cm}(\neq)} \ar[d] &  f \ar[d] \\
    C \ar[r]^{\phi|_C} \ar@{}[u]|\bigcup \ar[d]_{\mathrm{red}} & C \ar@{}[u]|\bigcup \ar[d]^{\mathrm{red}} & g|_C \ar[rr] \ar[d] && (g\circ \phi^{-1})|_C  \ar@{=}[r] \ar[d] & f|_C \ar[d] \\
    \m R  \ar[r]^{\phi_R} & \m R & g_R \ar[rr] && g_R\circ \phi_R^{-1}  \ar@{=}[r] & f_R 
  }
\end{align*} 

Our proof will be based on the use of Lagrangian spectral invariants applied to graphs of symplectic maps. Given a Hamiltonian function $H$ on a standard symplectic  torus $\bb T^{2n}$, the graph of its time--1 map $\phi_H^1$ is a Lagrangian submanifold of $\overline{\bb T^{2n}}\times \bb T^{2n}$. This graph is the image of the diagonal by the time--1 map of the Hamiltonian function 
$0\oplus H$ on $\overline{\bb T^{2n}}\times \bb T^{2n}$ given by $(0\oplus H)_t(q,p;Q,P)=H_t(Q,P)$. It will be convenient for us to see these Lagrangians as deformations of a standard ``coordinate'' Lagrangian subtorus rather than as deformations of the diagonal in $\overline{\bb T^{2n}}\times \bb T^{2n}$. Therefore we introduce the following two symplectic identifications:
\begin{align*}\Psi:\overline{\bb T^{2k_1+2k_2}}\times\bb T^{2k_1+2k_2} &\to \bb T^{2k_1}\times \bb T^{2k_1}\times \bb T^{2k_2}\times \bb T^{2k_2},\\
(q_1,p_1,q_2,p_2;Q_1,P_1,Q_2,P_2) &\mapsto \\
(q_1,P_1-p_1&;P_1,q_1-Q_1;q_2,P_2-p_2;P_2,q_2-Q_2), \\[8pt]
\mbox{and }\quad \Psi_R:\overline{\bb T^{2k_1}}\times\bb T^{2k_1} &\to \bb T^{2k_1}\times \bb T^{2k_1},\\
(q_1,p_1;Q_1,P_1) &\mapsto (q_1,P_1-p_1;P_1,q_1-Q_1).
\end{align*}

We see that $\Psi$ sends the diagonal of $\overline{\bb T^{2k_1+2k_2}}\times\bb T^{2k_1+2k_2}$ to the Lagrangian subtorus 
$$ L_0 = \bb T^{k_1} \times \{ 0 \} \times \bb T^{k_1} \times \{ 0 \} \times \bb T^{k_2} \times \{ 0 \} \times \bb T^{k_2} \times \{ 0 \} \,,$$
already introduced in Section \ref{sec:intro_tech}, and $\Psi_R$ sends the diagonal of $\overline{\bb T^{2k_1}}\times\bb T^{2k_1}$ to the Lagrangian subtorus
$$ L_R = \bb T^{k_1} \times \{ 0 \} \times \bb T^{k_1} \times \{ 0 \}.$$

The proof of Theorem \ref{theo:main-theo} will consist of a series of equalities and inequalities between spectral invariants. These identities are organized in four claims that we now give.
 
 The first claim follows immediately from Proposition \ref{prop:comp-lagr-ham-si} and the naturality of Lagrangian spectral invariants \eqref{eq:naturality-SI-2}. For example, \eqref{eq:3} below is due to the fact that
$$c_+(f_R)  = \ell([\Delta];\Delta,\Delta;0\oplus f_R) =  \ell([L_R];L_R,L_R;(0\oplus f_R)\circ\Psi_R^{-1})$$ with $\Delta$ the diagonal of $\overline{\bb T^{2k_1}}\times\bb T^{2k_1}$.

\begin{claim*} 
\begin{align}
c_+(f_R) &= \ell([L_R];L_R,L_R;(0\oplus f_R)\circ\Psi_R^{-1})\, ,\label{eq:3}\\
c_+(g_R) &= \ell([L_R];L_R,L_R;(0\oplus g_R)\circ\Psi_R^{-1})\, ,\label{eq:4}\\
c_+(g) &= \ell([L_0];L_0,L_0;(0\oplus g)\circ\Psi^{-1})\,,\label{eq:5}\\
c_+(g\circ\phi^{-1}) &= \ell([L_0];L_0,L_0;(0\oplus g\circ\phi^{-1})\circ\Psi^{-1})\,.\label{eq:6}
\end{align}
\end{claim*}

The next statement gives a relation between the spectral invariants of Lagrangians and functions defined on different spaces. This will be based on the splitting formula (\ref{eq:prdct_form}). As in Section \ref{sec:intro_tech}, we denote
\begin{align*}
  L_1 = \bb T^{k_1} \times \{ 0 \} \times \bb T^{k_1} \times \{ 0 \} \times \bb T^{k_2} \times \{ 0 \} \times \{ 0 \} \times \bb T^{k_2}\,.
\end{align*}
 We also recall that $L_0$ and $L_1$ split in the following form (see Section \ref{sec:intro_tech}):
$$L_0=L\times L_0'\quad\text{and}\quad L_1=L\times L_1'.$$

\begin{claim*} 
\begin{align}\ell([L_R];L_R,L_R;(0\oplus f_R)\circ\Psi_R^{-1}) &= \ell([L];L_0,L_1;(0\oplus f)\circ\Psi^{-1})\,,\label{eq:7}\\
 \ell([L_R];L_R,L_R;(0\oplus g_R)\circ\Psi_R^{-1}) &= \ell([L_0];L_0,L_0;(0\oplus g)\circ\Psi^{-1})\,.\label{eq:10}
\end{align}
\end{claim*}
\begin{proof} The Lagrangians $L_0$ and $L_1$ both contain $L_R$ and moreover can be decomposed in the form
\begin{align*}
  L_i = \underbrace{L_R \times \Lambda}_{\qquad = \, L \;\subset\; \bb T^{2k_2}\times\bb T^{2k_2}\times \bb T^{2k_1}} \;\times\; \underbrace{L'_i}_{\quad \subset\; \bb T^{2k_2}} \quad \mbox{ for $i=0$ and $1$}
\end{align*}
where $\Lambda=\bb T^{k_2}\times\{0\}\subset \bb T^{2k_2}$.
Then, if we denote $F=(0\oplus f)\circ\Psi^{-1}$ and $F_R=(0\oplus f_R)\circ\Psi_R^{-1}$, we see that we can decompose $F$ according to this spliting: $F=F_R\oplus 0\oplus 0$, where both $0$'s are seen as functions on $\bb T^{2k_2}$. By Theorem \ref{theo:prdct_form_l0l1},
$$\ell([L];L_0,  L_1;F)= \ell([L];L,L; F_R\oplus 0) + \ell([pt];L_0',L_1';0).$$
The second term on the right hand side vanishes. We may then apply the splitting formula (\ref{eq:prdct_form_l0=l1}):
\begin{align*}
\ell([L];L_0,  L_1;F) &= \ell([L];L,L; F_R\oplus 0)\\
&= \ell([L_R];L_R,L_R;F_R)+\ell([\Lambda];\Lambda,\Lambda;0),
\end{align*}
where again the second term vanishes. 
This proves Equation (\ref{eq:7}).

To prove Equation (\ref{eq:10}), one only needs to replace the pair $(L_0,L_1)$ by $(L_0,L_0)$, the pair $(L_0',L_1')$ by $(L_0',L_0')$, the function $f$ by $g$ and the function $f_R$ by $g_R$, and repeat the same argument. 
\end{proof}

We will also need the following equality which is essentially a manifestation of the fact that at any fixed time, the functions $f$ and $g\circ\phi^{-1}$ are constant on the leaves of the coisotropic submanifold $C$ and coincide on it.

\begin{claim*}
  \begin{align}
    \ell([L];L_0,L_1;(0\oplus f)\circ\Psi^{-1})=\ell([L];L_0,L_1;(0\oplus g\circ\phi^{-1})\circ\Psi^{-1}) \,. \label{eq:8}
  \end{align}
\end{claim*}

\begin{proof} Denote $F=(0\oplus f)\circ\Psi^{-1}$ and $G=(0\oplus g\circ\phi^{-1})\circ\Psi^{-1}$. Since $f$ and $g\circ\phi^{-1}$ coincide on $C$, the continuous functions $F$ and $G$ coincide on the coisotropic submanifold
$$W=\Psi(C\times C)\subset\bb T^{2k_1}\times\bb T^{2k_1}\times\bb T^{2k_2}\times\bb T^{2k_2}.$$
Observe that the Lagrangian $L_1$ is contained in $W$.

Since $f$ and $g$ are the respective lifts of $f_R$ and $g_R$,   their restriction to each leaf of $C$ only depends on the time variable.
Since $\phi^{-1}$ preserves the characteristic foliation of $C$, the function $g\circ\phi^{-1}$ is also a function of time on each leaf of $C$.
From this we deduce that $F$ and $G$ are functions of time on each characteristic leaf of $W$.
 
Now let $(F_k)_{k\in\bb N}$, $(G_k)_{k\in\bb N}$ be sequences of smooth Hamiltonians which uniformly converge  to $F$ and $G$, respectively, with the additional property that for all $k\in \bb N$, $F_k$ and $G_k$ are functions of time on each leaf of $W$
and $F_k-G_k=0$ on $W$. Such sequences can be constructed as follows. Let $F_k'$, $G_k'$ be two sequences of Hamiltonians that converge uniformly to $F$, $G$.
Note that the restrictions of $F$ and $G$ coincide on $W$ and are functions of time on each leaf hence admit the same reduced function $H$. Let $H_k$ be a sequence of Hamiltonians on the reduced space of $W$ which converges to $H$ uniformly. Each function $H_k$ can be lifted to a function $H_k'$ defined on $W$. By construction the functions  $F_k'-H_k'$ and $G_k'-H_k'$ converge to 0 on $W$. Denote by $F_k''$ and $G_k''$ their trivial extensions to $\bb T^{2k_1}\times\bb T^{2k_1}\times\bb T^{2k_2}\times\bb T^{2k_2}$, which also converge to 0. The functions $F_k=F_k'-F_k''$ and $G_k=G_k'-G_k''$ suit our needs: They converge respectively to $F$ and $G$ and they both coincide with the leafwise function $H_k'$ on $W$.

We will next show that  $\ell([L]; L_0, L_1; F_k) = \ell([L]; L_0, L_1; G_k)$ for all $k$.  The claim would then  follow by taking the limit of both sides as $ k \to \infty.$  Fix $k$ and let $H_r = rG_k +(1-r) F_k$ where $r \in [0,1]$. 
We will in fact prove the stronger statement that $\ell([L]; L_0, L_1; H_r)$ is a constant function of the variable $r$. 

For any $r$, $r' \in [0,1]$ the Hamiltonians $H_r$ and $H_{r'}$ are functions of time on each leaf of $W$ and $H_r = H_{r'}$ on $W$.  This is because the same statement is true for $F_k$ and $G_k$. It is not hard to check that this implies that for any point $p \in W$ and any $t \in [0,1]$ we have:
\begin{equation}\label{eq:leaf_flow}
\phi^t_{H_r}(p) \text{ and } \phi^t_{H_{r'}}(p) \text{ belong to the same characteristic leaf of } W. \end{equation}
Now, consider a critical point of the action functional $\m A^{L_0,L_1}_{H_r}$: It is a Hamiltonian chord $\phi^t_{H_r}(p)$ where $ p \in L_0$ and $\phi^1_{H_r}(p) \in L_1.$  The Hamiltonian $H_r$ is a function of time on characteristic leaves and so its  flow $\phi^t_{H_r}$ preserves $W$.  Since $\phi^1_{H_r}(p) \in L_1 \subset W$, we conclude that $\phi^t_{H_r}(p) \in W$ for all $t \in [0,1]$.  Using \eqref{eq:leaf_flow}, we see that $\phi^t_{H_{r'}}(p) \in W$ for any $t, r' \in [0,1]$.  Furthermore, \eqref{eq:leaf_flow} implies that $\phi^1_{H_{r'}}(p) \in L_1$:  This is because the Lagrangian $L_1 \subset W$ and hence any characteristic leaf of $W$ which intersects $L_1$ is entirely contained in $L_1$.  We conclude from the above that $t \mapsto \phi^t_{H_{r'}}(p)$ is a critical point of the action functional $\m A^{L_0,L_1}_{H_{r'}}$ and so there exists a bijection between the critical points of the two action functionals.
 
 Next, we will show that the two chords $\phi^t_{H_r}(p)$ and $\phi^t_{H_{r'}}(p)$ have the same action.  Since $H_r$ and $H_{r'}$ coincide on the leaves of $W$ we see, using \eqref{eq:leaf_flow}, that $H_r(\phi^t_{H_r}(p)) = H_{r'}(\phi^t_{H_r'}(p)).$  Hence, to show that the two chords have the same action we must prove that any two cappings  $u_r$ of  $\phi^t_{H_r}(p)$ and  $u_{r'}$ of $\phi^t_{H_{r'}}(p)$ have the same symplectic area. We will prove this using \eqref{eq:leaf_flow} as well.  Suppose that $r < r'$.  Fix any choice of $u_r$ and define $u_{r'} = u_r \# v$, where $v:[r, r'] \times [0,1] \to \bb T^{2n} \times \bb T^{2n}$ is defined by: $v(s,t) = \phi^t_{H_s}(p).$  We must show that the symplectic area of $v$ is zero:   Note that \eqref{eq:leaf_flow} implies that for any fixed $t$  the path $s \mapsto \phi^t_{H_s}(p)$ is contained in the same characteristic leaf of $W$.  Therefore, $\frac{\partial v}{\partial s}$ is always tangent to the characteristic leaves of $W$.  This combined with the fact that the image of $v$ is contained in $W$ yields that $\omega_{\mathrm{std}}(\frac{\partial v}{\partial s}, \frac{\partial v}{\partial t}) = 0$.  Hence, $v$ has zero symplectic area and the symplectic area of $u_r$ coincides with that of $u_{r'}$. 
 
 We conclude from the previous two paragraphs that $\Spec(H_r) = \Spec(H_{r'})$ for any $r$, $r' \in [0,1]$.  Recall that the spectrum of any Hamiltonian has measure zero.  We see that $r \mapsto \ell([L]; L_0, L_1; H_r)$ is a continuous function taking values in a measure zero set and thus it must be constant.  This finishes the proof.
\end{proof}
Finally, the following claim is a direct application of Proposition \ref{cor:triangle-inequality-1}.

\begin{claim*}\begin{align}\ell([L];L_0,L_1;(0\oplus g\circ\phi^{-1})\circ\Psi^{-1})\leq \ell([L_0];L_0,L_0;(0\oplus g\circ\phi^{-1})\circ\Psi^{-1}) \,.\label{eq:9}
  \end{align}
\end{claim*}

\noindent\textbf{End of the proof of Theorem \ref{theo:main-theo}.}
We now gather the identities collected in the above claims. Using the fact that $\phi^{-1}$ preserves $c_+$, we obtain:
\begin{align*}
c_+(f_R) &\stackrel{~(\ref{eq:3})}{=} \ell([L_R];L_R,L_R;(0\oplus f_R)\circ\Psi_R^{-1})\\
&\stackrel{~(\ref{eq:7})}{=} \ell([L];L_0,L_1;(0\oplus f)\circ\Psi^{-1})\\
&\stackrel{~(\ref{eq:8})}{=} \ell([L];L_0,L_1;(0\oplus g\circ\phi^{-1})\circ\Psi^{-1})\\
&\stackrel{~(\ref{eq:9})}{\leq} \ell([L_0];L_0,L_0;(0\oplus g\circ\phi^{-1})\circ\Psi^{-1})\\
&\stackrel{~(\ref{eq:6})}{=} c_+(g\circ\phi^{-1})\\
&\ \,= c_+(g)\\
&\stackrel{~(\ref{eq:5})}{=} \ell([L_0];L_0,L_0;(0\oplus g)\circ\Psi^{-1})\\
&\stackrel{~(\ref{eq:10})}{=} \ell([L_R];L_R,L_R;(0\oplus g_R)\circ\Psi_R^{-1})\\
&\stackrel{~(\ref{eq:4})}{=} c_+(g_R).
\end{align*}
Switching the roles played by $f_R$ and $g_R$ yields the reverse inequality $c_+(g_R)\leq c_+(f_R)$. Hence, $c_+(f_R)=c_+(g_R)$.

%
%

\bibliographystyle{abbrv}
\bibliography{biblio}

\end{document}